\documentclass[11pt]{article}
\usepackage{color,amsmath, amsthm, amsfonts, amssymb, dsfont, graphicx, listings, graphics, epsfig, enumerate, bbm,dsfont,stmaryrd
}

\usepackage[nodisplayskipstretch]{setspace}
\usepackage{etoolbox}

\makeatletter
\def\thm@space@setup{%
	\thm@preskip=0ex 
	\thm@postskip=\thm@preskip 
}
\makeatother

\usepackage{titlesec} 
\titlespacing{\section}{0pt}{2ex}{1ex}
\titlespacing{\subsection}{0pt}{1ex}{0ex}
\titlespacing{\subsubsection}{0pt}{0.5ex}{0ex}
\usepackage{multirow}

\usepackage[margin=1in]{geometry}
\usepackage[fleqn,tbtags]{mathtools}
\usepackage[english]{babel}

\usepackage[round]{natbib}   
\setcitestyle{aysep={}} 
\setcitestyle{citesep={,}} 

\usepackage{mathrsfs}  
\usepackage{comment} 
\usepackage{enumitem} 
\setlist[itemize]{noitemsep, topsep=0pt} 
\setlist[enumerate]{noitemsep, topsep=0pt} 
\usepackage{color,soul} 

\usepackage{braket} 
\usepackage{xcolor}
\usepackage{accents} 
\usepackage[font={small}]{caption} 
\usepackage{booktabs} 
\usepackage{algorithm} 
\usepackage{algorithmic} 
\usepackage{tikz} 
\usepackage{proof-at-the-end}
\usepackage{placeins} %
\usepackage{slashbox} 
\usepackage{diagbox} 

\usepackage{versions}
\excludeversion{dem} 
\allowdisplaybreaks 

\addto\captionsenglish {}
\makeatletter \renewenvironment{proof}[1][\proofname] {\par\pushQED{\qed}\normalfont\topsep6\p@\@plus6\p@\relax\trivlist\item[\hskip\labelsep\bfseries#1\@addpunct{.}]\ignorespaces}{\popQED\endtrivlist\@endpefalse} \makeatother
\newtheorem{thm}{Theorem}[section]
\newtheorem{cor}[thm]{Corollary}
\newtheorem{lem}[thm]{Lemma}
\newtheorem{prop}[thm]{Proposition}

\theoremstyle{definition}

\newtheorem{rem}[thm]{Remark}

\numberwithin{equation}{section}

\newcommand*{\dt}[1]{%
	\accentset{\mbox{\large\bfseries .}}{#1}}

\newlength{\dhatheight}

\newcommand*\circled[1]{\tikz[baseline=(char.base)]{
			\node[shape=circle,draw,inner sep=1.5pt] (char) {#1};}} 

\newcommand{\EE}{\mathbb{E}}
\newcommand{\RR}{\mathbb{R}}
\newcommand{\PP}{\mathbb{P}}
\newcommand{\NN}{\mathbb{N}}
\newcommand{\ZZ}{\mathbb{Z}}
\newcommand{\FF}{\mathbb{F}}
\newcommand{\TT}{\mathbb{T}}

\newcommand{\II}{\mathbb{I}}

\newcommand{\cC}{\mathcal{C}}

\newcommand{\cF}{\mathcal{F}}

\newcommand{\cH}{\mathcal{H}}

\newcommand{\cT}{\mathcal{T}}

\newcommand{\tG}{\widetilde{G}}
\newcommand{\tV}{\widetilde{V}}
\newcommand{\tS}{\widetilde{S}}
\newcommand{\tJ}{\widetilde{J}}
\newcommand{\tpi}{\widetilde{\pi}}

\newcommand{\E}{\mathbb{E}}

\newcommand{\of}[1]{\ensuremath{\left( #1 \right)}}
\newcommand{\cb}[1]{\ensuremath{ \left\{ #1 \right\} }}
\newcommand{\sqb}[1]{\ensuremath{ \left[ #1 \right] }}

\def\prehp(#1,#2){\ensuremath{  #1 \cdot #2 }}

\newcommand{\condition}[1]{\ensuremath{ \left. #1 \,\right|\, }} 

\title{End-of-Life Inventory Management Problem: \\ Results and Insights
}
\author{
Emin Ozyoruk \thanks{Booth School of Business, University of Chicago, Chicago, Illinois, 60637, USA, eozyoruk@chicagobooth.edu.}
\and
Nesim Kohen Erkip \thanks{Department of Industrial Engineering, Bilkent University, Ankara, 06800, Turkey, nesim@bilkent.edu.tr.}
\and
\c{C}a\u{g}{\i}n Ararat\thanks{Department of Industrial Engineering, Bilkent University, Ankara, 06800, Turkey, cararat@bilkent.edu.tr.}
}

\begin{document}
\maketitle
\setlength{\belowdisplayskip}{1pt} \setlength{\abovedisplayskip}{1pt} 

\begin{abstract}
	\noindent We consider a manufacturer who manages the end-of-life phase and takes one of the three actions at each period: (1) place an order, (2) use existing inventory, (3) stop holding inventory and use an outside/alternative source. Two examples of this source are discounts for a new generation product and delegating operations. Demand is described by a non-homogeneous Poisson process, and the decision to stop holding inventory is described by a stopping time. After formulating this problem as an optimal stopping problem with additional decisions and presenting its dynamic programming algorithm, we use martingale theory to facilitate the calculation of the value function. Moreover, we show analytical results to understand the additional difficulties of the problem solved, as well as structural results on optimal stopping times. Furthermore, we devise an expandable taxonomy and categorize the models in the literature. Analytical insights from the models as well as an extensive numerical analysis show the value of our approach. The results indicate that the loss can be high in case the manufacturer does not exploit flexibility in placing orders or use an outside source. Several managerial insights are obtained through numerical analysis as well as structural results to facilitate decision-making during the end-of-life horizon.
\\
\\[-10pt]
\textbf{Keywords and phrases:} Management of end-of-life problem, end-of-life inventory management, optimal stopping with additional decisions, non-homogeneous Poisson process \\
\\[-10pt]
\end{abstract}

\section{Introduction and Literature Review}\label{intro}
\subsection{Motivation and Literature Review}
While rapid technological developments have been shortening the life-cycle of products sold in the market, competition and customer satisfaction have made the firms increase the warranty periods of those products. To fix a product in case of failures, a firm holds spare parts inventory for long periods and even after the product is no longer produced. This leads to a challenging problem of inventory management of spare parts in the end-of-life phase – a time frame within the product’s life-cycle that begins when the product is no longer produced and that ends at the expiration date of all customers’ warranties  \citep{fortuin1980all}. Original equipment manufacturers strive to properly manage the inventory in the end-of-life phase since spare parts are held for long periods although the demand for them can be quite low. For instance, in electronics industry, a manufacturer may need to keep spare parts for 4-30 years after the product is discontinued from production \citep{teunter2002inventory}. It might seem tempting to pile up abundant inventory to obey customer warranties; however, this may result in excessive holding and scrapping costs given that the demand is expected to be low. Indeed, HP suffered from huge obsolescence costs due to end-of-life write-offs \citep{callioni2005inventory}, and in general, after-sales services can be a significant source of profit for the firms \citep{shi2019optimal}. As a result, several strategies have been developed to control inventory and mitigate the risk of over- and under-stocking of spare parts in the end-of-life phase. 

Early approaches for inventory control in the end-of-life phase attempt to use classical inventory models while aiming to calibrate the parameters pertinent to this phase. For instance, \citet[p.~363, Subsection\@ 8.5.1]{silver2016inventory} review the studies that develop extensions of the economic order quantity (EOQ) model while assuming a deterministic and decreasing demand rate. Those studies find the number of replenishments to make as well as the timing and sizes of these replenishments. Simultaneously, several studies are motivated by the intermittent demand structure in this phase, devising inventory models with stochastic demand. Extensions of the newsvendor model, for example, are developed where the parameters (e.g., mean and standard deviation of demand) are estimated from available data. Such studies are reviewed by \citet[p.~364, Subsection 8.5.2]{silver2016inventory} as well. 

The practically oriented approaches often assume that the original equipment manufacturer can place a single order at the beginning of the end-of-life phase, and they propose complementary business strategies. The motivation behind the single-order assumption is that a component manufacturer might decide to stop producing certain spare parts, thereby requiring the original equipment manufacturer to place a final order. This final order is also called last-time buy, final buy, end-of-life buy, or buy all-time requirements. On the other hand, complementary strategies aim to support the final buy in case of a discrepancy between the realized demand and the order quantity. The wide literature on business strategies complementing a final order includes, but is not limited to, repairing defective spare parts collected from customers \citep{behfard2015last,behfard2018last} while repairing may not be feasible for some of them \citep{van2009final}, buying back functional or dysfunctional used products to take them apart and obtain the recoverable spare parts \citep{pourakbar2014Phaseout,kleber2012dynamic}, considering budget constraints \citep{hur2018end} or multiple spare parts in the bill-of-materials of a main product \citep{bradley2009lifetime}, extending customer contracts \citep{pincce2015role, leifker2014determining}, designing a new product to replace the obsolete one (design refresh) \citep{shen2014modeling,shi2020optimal}, partially scrapping spare parts in case of over-stocking \citep{pourakbar2012end}, differentiating customers based on demand criticality or service contracts \citep{pourakbar2012customer}, re-manufacturing \citep{shi2019optimal,bayindir2007assessing}, finding outside/alternative sources \citep{pourakbar2012end,frenk2019exact,van2013last,jack2000optimal}, and finally, obviating the need to place a final order at time zero \citep{cattani2003good,teunter2002inventory,pincce2011inventory}. Common to all the studies above is the fact that the end-of-life management problem considered is more than an inventory problem and hence several actions should be considered simultaneously.

In this study, we focus on the last two strategies described above. The benefit of a complementary strategy that finds an outside/alternative source, instead of holding spare parts inventory, can be two-fold. On the one hand, in case the demand for spare parts exceeds the available inventory, the manufacturer can start using the outside source as a back-up source and avoid underage costs. On the other hand, it can be used to get rid of excess inventory in case of insufficient demand, decreasing overage costs. Some examples of this outside/alternative source can be to purchase expedited spare parts supply from a third-party supplier, to replace the failed product with a new generation product \citep{pourakbar2012customer,frenk2019exact, van2013last,jack2000optimal}, or to substitute with another spare part having the same functionality \citep{shi2020optimal}. If the cost of such a source decreases over time (for instance, due to price erosion of a new generation product), then this strategy can become truly valuable. 

Within the literature incorporating an outside/alternative source, \citet{pourakbar2012end} consider a manufacturer who places a final order at time zero and can decide to use an outside/alternative source at each of the forthcoming time periods. \citet{frenk2019exact} assume that the manufacturer makes a static decision (at time zero) on the final order quantity and on the time to stop holding inventory (called switching time). \cite{frenk2019optimal} extends the model in \cite{frenk2019exact} with more general parameters and describe the decision to stop holding inventory by a stopping time, solving an optimal stopping problem by means of a dynamic programming algorithm. \cite{shi2020optimal} consider a design refresh program that substitutes an obsolete part with an alternative part, while modeling this problem as a two-stage stochastic dynamic program. 

We could find a few recent studies which analyze the benefit of providing flexibility in placing orders in the end-of-life phase. We consider two such flexibilities: timing of the final (or first) order, and having multiple orders. Regarding the first flexibility, it is reasonable to accept the existence of a time point at which the manufacturer places a final order. Still, such a time point may need to be found after completing an in-depth analysis since, after all, the component manufacturers might be willing to produce the spare parts as long as it is profitable to do so. Among the studies allowing flexibility in placing orders in the end-of-life phase, \citet{cattani2003good} analyze the effects of delaying a final order rather than placing it at time zero, and determine the optimal timing of the final buy from an aggregated supply chain perspective by including both the manufacturer and the supplier. They also characterize the benefit of delaying a final order under different demand scenarios. The second flexibility we consider is the possibility of multiple orders during the end-of-life horizon. There are numerous examples in the early and more recent literature -- here we present a few. \citet{inderfurth2008decision} devise a dynamic programming model to help manufacturers who can place extra production/procurement orders as well as remanufacture the recoverable spare parts. \citet{inderfurth2013advanced} further explore the problem studied by \citet{inderfurth2008decision}, and devise an advanced heuristic that provides near-optimal solutions and that can quickly solve real-life problem instances. \citet{teunter2002inventory} devise a continuous-time solution when demand is described by a Poisson process (with constant rate) and find an optimal base-stock policy where order-up-to levels decrease over a finite time horizon. \citet{pincce2011inventory} also provide a continuous-time formulation; their model mainly differs from \citet{teunter2002inventory} in that partial obsolescence is allowed, that is, intensity rate drops to a lower level at a known future time instance. Also see \citet{david1997dynamic} for a dynamic programming approach when demand is deterministic. 

Finally, we review the components of the objective functions considered. As expected, most of the objective functions of the above studies are related to costs but are varying in type and in the timing of charging costs. However, many of them assume that costs are charged at discrete time points \mbox{\citep{shi2020optimal}}. Additionally, many of them assume that there is no fixed cost of ordering (understandable since one ordering instance is allowed) \mbox{\citep{pourakbar2012end}}. Hence, having non-zero inventory at time zero would not have a significant effect on the solutions proposed. Many of those studies, by construction, assume lost-sales for the excess demand or the demand is satisfied from outside sources \mbox{\citep{frenk2019exact}}, though some of them allow for backordering and penalize both the time and units backordered \mbox{\citep{teunter2002inventory}}. Finally, some of the costs are assumed to be constant over time; exceptions are allowing for discounting, a decreasing cost of alternative source over time \mbox{\citep{frenk2019exact}}, or an increasing unit procurement cost after time zero \mbox{\citep{teunter2002inventory}}. 

In our approach, we consider an end-of-life problem where one makes use of the flexibilities considered above while properly considering their costs (fixed cost per order, as well as others) simultaneously.

\subsection{Novelty of the Approach and Contributions}

This study analyzes the value of providing flexibility in placing orders while making use of strategies related to the end-of-life phase. In more general terms, we consider a multi-period, lost-sales inventory problem where lost sales can be compensated by an outside source, as well as the outside source can become the main source if we decide so (modeled by a stopping time). The novelty of our study is the incorporation of all the following features during modeling and solution stages.
\begin{itemize}[leftmargin=*]
	\item \textit{Flexibility in ordering:} Instead of being required to place a single order at time zero, the manufacturer has the flexibility to place orders at any time and can limit the number of orders. Additionally, a time point can be found at which the manufacturer does not choose to place any orders afterward, i.e., the timing of the final order. All these aspects bring in several flexibilities: Allowing for multiple orders, chances of delaying the first order and limiting the number of orders.
	\item \textit{Strategic switch to an alternative/outside source:} The manufacturer can stop holding inventory and use an alternative/outside source to satisfy demand (strategically decide not to use internal sources). This source has a relatively high per-unit cost; however, it can be useful in avoiding excessive penalty and holding costs in the future. Also, the manufacturer does not need to put an effort into the use of this source. Such a source can be another option in cases where redesigning products or spare parts may cannibalize design resources that could otherwise be used for designing new products \citep{bradley2008product}. The manufacturer’s decision to stop holding inventory is described by a stopping time.
	\item \textit{Demand variability:} The demand for spare parts is described by a non-homogeneous Poisson process. We also assume a non-increasing intensity function to fit the problem description, though for most of our results such an assumption is not necessary. 
	\item \textit{Cost components:} We consider fixed ordering costs as well as non-stationary costs. Additionally, we compute expected costs when these are charged continuously (rather than at discrete time points). This allows us to naturally operate with periods which are not necessarily spaced equally in time. Another motivation for the continuous-time calculation of costs is that the manufacturer may not be able to review the inventory for long periods, so we may miss correct representation of costs. For instance, in our model, we describe the exact time that the inventory on hand hits zero and hence lost-sales is observed by using a stopping time.
	\item \textit{Decision structure and solution methodology:} The manufacturer's problem is to make one of the three decisions at each period: (1) place an order for spare parts, (2) do nothing and use existing inventory to satisfy demand, or (3) stop holding inventory permanently and use outside/alternative source until the end. We cast this combined inventory control and optimal stopping problem as an ``optimal stopping problem with additional decisions'' that can be solved by means of a stochastic dynamic programming (DP) algorithm (see \cite{oh2016characterizing} for the definition of optimal stopping problems with additional decisions). Note that DP might be considered as an expensive tool for solution for tactical and operational problems. However, most of the typical end-of-life problems are expected to warrant the cost of the solution approach -- which is to be paid once.
\end{itemize}

Below we summarize our main contributions to the end-of-life inventory management problem.
\begin{itemize}[leftmargin=*]
	\item We propose a flexible, expandable taxonomy for the end-of-life inventory management problem which facilitates grouping the existing studies based on their available decisions.
	\item We propose a general framework for the problem by combining several issues raised in the literature and solving it using dynamic programming. We numerically show its benefits over the other approaches.
	\item We allow continuous cost computations by using martingale theory to facilitate the calculation of the value function, allowing the decision maker to adjust the period definition, as well as the use of non-equal time periods without compromising the exactness of the cost terms computed.
	\item We provide some structural insights on the problem, as well as on some special cases. Specifically, we present the relation of the model with the inventory literature and utilize several properties of the optimal stopping time under various considerations to support decision-making during the horizon. Under the premise that an optimal strategy (obtained with stochastic DP) is implemented, we obtain the distribution of stopping times to further support communication with the outside source provider. 
	\item We discuss important managerial insights that will likely support decision-making throughout the life cycle of the problem. (1) Some of these insights follow extensive computations. Remarks are given on the parameter settings for which the flexibilities in the model bring substantial savings, allowing decision makers to focus. (2) Additional insights follow the structural results on stopping times -- mainly, these results may be shared with the supplier of the alternative source to ensure smooth transition of service operations to the supplier. (3) Finally, following numerical solutions, several managerial insights are proposed to control the parameter settings of the end-of-life problem. Specifically, the effects of controlling customer arrival rates, making monetary arrangements to support development of the alternative source and extending the warranty period to decrease the expected cost of the end-of-life period are considered so that the outcome becomes more attractive for the manufacturing firm.
\end{itemize}

The rest of this paper is organized as follows. Section \ref{Subsection: Representation of Expected Costs} formally defines the costs considered and their computations, whereas Section \ref{Subsection: D8F DP formulation} defines our problem and presents the main DP algorithm. Section \ref{Subsection: Taxonomy} presents a taxonomy for related problems. Section \mbox{\ref{Section: Structural Results}} provides the analytical results to run the DP algorithms as well as several structural results on stopping times. In Section \mbox{\ref{Section: Numerical Analysis}}, we provide extensive numerical computations to analyze the advantages of our approach, as well as to come up with several managerial insights. Finally, in the last section, we conclude, and present practical implications as well as possible extensions.

\section{Problem Definition and Details} \label{Section: Problem Definition}
Let $(\Omega, \cH, \PP)$ be a probability space and let $T\in[0,\infty)$. We start by assuming that the demand for spare parts is described by a non-homogeneous Poisson process $N = (N_t)_{t\in[0,T]}$ with an intensity function $\lambda\colon[0,T]\rightarrow\RR_+$ and the associated mean value function $t\mapsto \Lambda(t):=\int_{0}^{t}\lambda(u)du$. We assume that $\lambda$ is right-continuous with left-limits, piece-wise smooth (that is, differentiable except at finitely many points) and non-increasing. Most of the results in this study can be recovered without the assumption that $\lambda$ is non-increasing. Still, such assumption can be more appropriate to describe the demand for spare parts in the end-of-life phase. The manufacturer periodically reviews the inventory level and for brevity of notation, we assume that the lengths of time periods are identical. Our model can be easily adjusted for non-identical period lengths. At each time point $k\in\TT := \{0,1,\dots,T\}$, the manufacturer observes the current inventory level $x\in\ZZ_+:=\{0,1,\ldots\}$ and decides whether to stop or continue holding inventory. We sometimes refer the time interval $[k,k+1]$ as the $k^{\text{th}}$ period.

\subsection{Representation of Expected Costs}\label{Subsection: Representation of Expected Costs}
After observing inventory $x\in\ZZ_+$ at time $k\in\TT$, the manufacturer may decide to continue holding inventory. In this case, an order $\mu_k(x)\in \ZZ_+$ can be placed, where the function $\mu_k\colon\ZZ_+\rightarrow\ZZ_+$ specifies the order amount. The order cost function $c\colon\ZZ_+\rightarrow\RR_+$ is given by
\begin{flalign}\label{Equation: Order cost function}
c(m) := \begin{cases} 
K + \bar{c}\,m, &\text{if } m>0, \\
0, & \text{if } m=0,
\end{cases}
\end{flalign}
where $\bar{c}\in\RR_+$ is the unit purchasing cost and $K\in\RR_+$ is the fixed ordering cost. We assume that the ordering cost at period $k$ is discounted by $e^{-\delta k}$ and the lead time is zero, where $\delta\in[0,1]$ is the discount rate of continuous compounding. After placing an order at time $k$, the manufacturer continues operating during the time interval $[k,k+1]$, and the holding cost accrues with rate $c_{1} \in \RR_+$. Hence, the expected value of inventory holding cost for the $k^{\text{th}}$ period is given by
\[
H(k,x):=c_{1} \EE\sqb{\int_{k}^{k+1} e^{-\delta (u-k)} (x -(N_u-N_k))^+ du}.
\]
The following lemma enables us to calculate the holding cost.

\begin{lem}
	\label{Lemma: Holding Cost}
	For every $k\in\TT$, the expected holding cost satisfies $H(k,0)=0$ when $x=0$, and
	\[
	H(k,x)= c_1 \sum_{n=0}^{x-1} \sum_{i=0}^{n} \int_{k}^{k+1} e^{-\delta(u-k)} e^{-(\Lambda(u)-\Lambda(k))} \frac{(\Lambda(u)-\Lambda(k))^i}{i!} \,du,
	\]
	when $x\geq 1$.  
\end{lem}
The proof of Lemma \ref{Lemma: Holding Cost} is in Appendix \ref{Section: Proofs of Section Problem Defn}.

If the inventory level hits zero during $[k,k+1]$ and a defective part arrives, then the manufacturer loses the opportunity to replace it  from the inventory. Instead, (similar to a lost-sales inventory system) the part is replaced by paying a time-dependent unit cost $c_2:[0,T]\rightarrow\RR_+$ at the time of arrival. The expected value of such replacement cost is given by
\[
L(k,x):=\EE \sqb{\int_{ (k+1)\wedge \sigma^{k}_{x}}^{k+1} e^{-\delta (u-k)} c_2(u) dN_u},
\] 
where $(k+1)\wedge\sigma^k_{x} = \min\cb{k+1,\sigma^k_x}$, and 
\[
\sigma^{k}_x:=\inf\{u>k: N_u-N_k\geq x\}
\]
denotes the arrival time of the $x^{\text{th}}$ item after time $k$. We denote $\sigma_x := \sigma^0_x$. The following lemma enables us to calculate $L(k,x)$. 
\begin{lem}
	\label{Lemma: Lost sales cost}
	For every $k\in\TT$ and $x\in\ZZ_+$, the expected replacement cost can be expressed as
	\begin{align*}
	L(k,x)
	=& \int_{k}^{k+1} e^{-\delta (u-k)} c_2(u) \lambda(u)\,du \\
	& - \sum_{i=0}^{x} \int_{k}^{k+1} e^{-\delta(u-k)} c_2(u) \lambda(u) e^{-(\Lambda(u)-\Lambda(k))}\frac{(\Lambda(u)-\Lambda(k))^i}{i!} \,du. 
	\end{align*}
\end{lem}

The proof of Lemma \ref{Lemma: Lost sales cost} is in Appendix \ref{Section: Proofs of Section Problem Defn}. 

\noindent We define the one-period operation cost for period $k$ and inventory level $x$ as
\begin{equation}\label{Equation: Defn of C(k,x)} 
C(k,x) := H(k,x) + L(k,x).
\end{equation}

On the other hand, after observing inventory $x\in\ZZ_+$ at time $k\in\TT$, the manufacturer may decide to stop. In this case, the available inventory is scrapped with unit cost $c_{4}\in\RR$ (if $c_4<0$, then we can interpret it as scrapping revenue). Future defective parts, if any, are replaced by using an alternative/outside source at a time-dependent unit cost given by a function $c_{3}:[0,T]\rightarrow\RR_+$. Therefore, the expected cost of stopping to hold inventory is given by
\begin{equation} \label{Equation: S(k,x_k)}
S(k,x) := c_{4}x + \E\sqb{\int_{k}^{T} e^{-\delta(u-k)} c_{3}(u) dN_u }.
\end{equation}
We assume that $\bar{c} > -c_4$ holds since otherwise, the manufacturer can place an arbitrarily large order and scrap inventory at the same time. Moreover, it is natural to assume that $c_2(u)\geq c_3(u)$ for every $u\in[0,T]$ so that the use of an outside source is meaningful.

\subsection{Dynamic Programming Formulation} \label{Subsection: D8F DP formulation}
Now, we are ready to pose our problem and its dynamic programming formulation. To that end, let $\cT$ denote the set of all stopping times $\tau:\Omega\rightarrow\TT$ of the filtration generated by the demand process $(N_t)_{t\in[0,T]}$. Let $\pi =(\tau, \mu_1,\mu_2,\dots,\mu_T)$ be a policy that specifies a stopping time $\tau$ and an order amount $\mu_k(x_k)$ for every time $k$ and inventory $x_k$. Let $\Pi$ denote the set of all (admissible) policies. Then, under an arbitrary policy $\pi\in\Pi$, the inventory level at time $k+1$ is equal to the inventory at $k$ plus order minus demand:
\begin{equation}\label{Equation: X_{k+1}= (X_k+mu - (N_{k+1}-N_k) )}
X^\pi_{k+1} = \Big(X^\pi_k + \mu_k(X^\pi_k) - (N_{k+1} - N_k )\Big)^+, \quad X^\pi_0 = x\in\ZZ_+,
\end{equation} 
where $(x)^+ := \max\cb{0,x}$. We use the notation $X_{k+1}$ for $X^\pi_{k+1}$ when $\pi$ is an arbitrary policy and there is no risk of confusion, for brevity of notation. The manufacturer's problem is to determine both the optimal order amounts and the optimal time to stop in order to minimize the total costs. We formulate this problem as
\begin{align} \label{Equation: V*(x)}
&V^*(x) = \inf_{\pi\in\Pi} \left. \E\left[\sum_{k=0}^{\tau-1} e^{-\delta k} \Big(c(\mu_k(X_k)) + C(k,X_k+\mu_k(X_k))\Big) + e^{-\delta\tau} S(\tau,X_\tau) \right\vert X_0 = x\right],
\end{align}
where $x\in\ZZ_+$ is called the initial inventory level and the summation over $k$ is set to zero if $\tau=0$. This formulation yields an optimal stopping problem with additional decisions \citep{oh2016characterizing} and can be solved by means of the following stochastic dynamic programming (DP) algorithm. Define the backward recursion for each $k\in\cb{T-1,T-2,\dots,0}$ and $x_k\in\ZZ_+$ by
\begin{align}
V(k,x_k) = & \min\cb{S(k,x_k), J(k,x_k) }, \label{Equation: V(k,x_k) DP Definition} \\
J(k,x) := & \inf_{m\in\ZZ_+} \cb{c(m) + G(k,x+m)}, \label{Equation: J(k,x)} \\
G(k,y) := & C(k,y) + e^{-\delta} \EE\sqb{ V\big(k+1, (y-(N_{k+1}-N_k))^+\big) }. \notag
\end{align}
Here, $G$ denotes the cost of continuing operations one more period, and $J$ denotes the continuation cost after finding the best order amount. The value function $V$ compares the stopping cost and the continuation cost. Also define the terminal condition by $
V(T,x_T) = S(T,x_T) = c_{4}x_T$.
Then, the DP formulation solves the manufacturer's problem in the sense that \mbox{$V(0,x)=V^*(x)$} for every $x\in\ZZ_+$. Moreover, the optimal order amount $\mu^*_k(x_k)$ attains the infimum in \eqref{Equation: J(k,x)}. Furthermore, an optimal stopping time $\tau^*$ is the one that stops the process if $S(k,x_k)\leq J(k,x_k)$, in other words, $\tau^*:= \min\{ k\in\TT \,:\, S(k,X^{\pi^*}_k)\leq J(k,X^{\pi^*}_k) \}$, where $\pi^*$ denotes an optimal policy. The next section provides related problems. 

\subsection{Taxonomy and Dynamic Programming Representation of Related Problems} \label{Subsection: Taxonomy}
In this subsection, we develop benchmark models for the main model in Subsection \ref{Subsection: D8F DP formulation}. To keep track of different formulations, we use the notation $V^{a/b/c}$ to describe the decisions allowed. In this way, we introduce levels of flexibility mentioned to be included or excluded in the problem as long as the actual environment allows. The symbol $a$ represents how the time for stopping is decided: $a=D$ means that the decision is dynamic and it is described by a stopping time; $a=S$ means that the decision is static and made at time zero (we call it switching time); $a=T$ means that we decide not to stop until the end of horizon. The symbol $b$ represents the number of orders allowed: If $b=M$ for some $M\in\ZZ_+$, then the manufacturer can place $M$ orders throughout the horizon; $b=\infty$ means that the manufacturer can place an order at each period with no restriction. Finally, the symbol $c$ represents alternatives with respect to the timing of the first order: $c=Z$ means that the first order must be placed at time zero and $c=F$ means that the manufacturer is free to place the first order at any time. Table \ref{Table: Notation for Benchmark Models} summarizes the notation. For instance, our main DP model in Section \ref{Subsection: D8F DP formulation} can be denoted by $D/\infty/F$. This notation is helpful to keep track of the decisions considered for different formulations while comparing them. In numerical analyses, we compare these models together with our main model to show the value of our approach. 

\begin{table}[h!]
	\centering
	\begin{tabular}{|c|c|}
		\hline
		$a$ & Time to Stop Holding Inventory  \\ \hline
		$S$ & Static decision -- made at time 0 \\
		$D$ & Decision is made dynamically \\
		$T$ & Do not stop until the end of horizon $T$ \\ \hline \hline
		$b$ & Maximum Number of Orders  \\ \hline
		$M$& Maximum $M\in\NN$ orders \\
		$\infty$ & Unrestricted number of orders \\ \hline \hline
		$c$ & Order Time \\ \hline
		$Z$ & First order must be placed at time zero \\
		$F$ & First order can be placed at any time \\
		\hline
	\end{tabular}	
	\caption{\label{Table: Notation for Benchmark Models} Notation for benchmark models. }
\end{table}

\subsubsection{D/1/F - Single Order Opportunity at Any Time and Stopping Time} \label{Subsection: D/1/F}
This benchmark dynamic programming formulation analyzes the effects of delaying a single order, and \citet{cattani2003good} present this idea in a different setting. Let $z\in\{0,1\}$ be the number of remaining orders that the manufacturer can place. The following dynamic programming algorithm describes this formulation. Define the terminal cost for each $x_T\in\ZZ_+$ and $z\in\{0,1\}$ by $V^{D/1/F}(T,x_T,z) = S(T,x_T).$ If $z=0$, then define the backward dynamic programming algorithm for each $k\in\cb{T-1,T-2,\dots,0}$ and $x_k\in\ZZ_+$ by
\begin{align*}
& V^{D/1/F}(k,x_k,0)
 = \min\cb{ S(k,x_k),\ C(k,x_k) + e^{-\delta} \E\sqb{V^{D/1/F}(k+1,\big(x_k - (N_{k+1}-N_k)\big)^+,0)} }.
\end{align*}
If $z=1$, then define the backward dynamic programming algorithm for each $k\in\cb{T-1,T-2,\dots,0}$ and $x_k\in\ZZ_+$ by
\begin{align*} 
& V^{D/1/F}(k,x_k,1)\\
&= \min \bigg\{ S(k,x_k),\  C(k,x_k) + e^{-\delta} \E\sqb{V^{D/1/F}(k+1,\of{x_k - (N_{k+1}-N_k}^+,1)}, \notag \\
& \quad\quad \textcolor{white}{\min\{ } \inf_{m\in\ZZ_+} \of{ c(m) + C(k,x_k+m)  + e^{-\delta} \E\sqb{V^{D/1/F}(k+1,\of{x_k+m- (N_{k+1}-N_k)}^+,0)} } \bigg\}.
\end{align*}
\subsubsection{D/1/Z - Single Order Opportunity at Time Zero and Stopping Time} \label{Subsection: D/1/Z}
A prevalent assumption in the literature is that a final order has to be placed at time zero. Therefore, we develop a dynamic programming algorithm to reflect the manufacturer's decision when only one order can be placed at time zero and the manufacturer can stop holding inventory at any time. This model resembles the one presented by \citet{pourakbar2012end}. Define the backward dynamic programming algorithm for each $k\in\cb{T-1,T-2,\dots,0}$ and $x_k\in\ZZ_+$ by
\begin{align*} 
& \dt{V}^{D/1/Z}(k,x_k) 
= \min\cb{S(k,x_k),  C(k,x_k) + e^{-\delta} \E\sqb{\dt{V}^{D/1/Z}(k+1,\big(x_k - (N_{k+1}-N_k)\big)^+)}}.
\end{align*}
Also define the terminal condition by $\dt{V}^{D/1/Z}(T,x_T) = S(T,x_T)$.
The optimal order quantity at time zero and the value of this dynamic program is found by calculating
\begin{equation}\label{Equation: V^{D/1/Z}(k,x_k)}
V^{D/1/Z}(x)=\inf_{m\in\ZZ_+} \of{c(m)+\dt{V}^{D/1/Z} (0,x+m)}, \quad x\in\ZZ_+.
\end{equation}
It is possible to see the following relation between $D/1/F$ and $D/1/Z$: While solving the model $D/1/F$, if we decide to place an order at time $k$, then we solve the model $D/1/Z$ with a different time horizon that is equal to $T-k$.

\subsubsection{S/$\infty$/F or S/1/F or S/1/Z} \label{Subsection: S/infty/F, S/1/F, S/1/Z} 
$S/\infty/F$ can be formulated as a special case of $D/\infty/F$. For each switching time $k\in\TT$, we implement a restricted version of the dynamic programming algorithm and select the best switching time $k^*$. Moreover, $S/1/F$ can be formulated as a special case of $D/1/F$ presented in Subsection \ref{Subsection: D/1/F}. We modify the value function $V^{D/1/F}$ by eliminating the stopping option with cost $\widetilde{S}(k,x_k)$ and solve the DP algorithm. Finally, $S/1/Z$ can be formulated as a special case of $D/1/Z$ presented in Subsection \ref{Subsection: D/1/Z}. For every $t\in\TT$, we implement $D/1/Z$ without being able to stop.

\subsubsection{T/$\infty$/F or T/1/F or T/1/Z}
These benchmark models are further special cases of $S/\infty/F$, $S/1/F$ and $T/1/F$. They resemble the classical inventory models which can be solved by means of standard DP algorithms. For instance, see \citet[Chapter~4]{beyer2010markovian}. 

\subsubsection{S/M/F or D/M/F for M$>$1}
These models can be formulated by using a similar idea as for $D/1/F$ in Subsection \ref{Subsection: D/1/F} by representing the number of setups as a state variable. Finally, Table \ref{Table: Benchmark Models and Related Study} relates the benchmark models with some of the more relevant literature mentioned before.

\begin{table}[h!]
	\centering
	\begin{tabular}{|l|l|l|}
		\hline
		Model & Explanation & Related Studies \\
		\hline
		$D/\infty/F$ & Multiple orders and stopping time & This study \\
		\hline
		\multirow{3}{*}{$D/1/Z$} & \multirow{3}{*}{Single order at time zero and stopping time} & \cite{frenk2019optimal} \\
		& &  \cite{pourakbar2012end} \\
		& & \cite{frenk2019order} \\  
		\hline
		$S/1/Z$ & Single order at time zero and switching time & \cite{frenk2019exact} \\    
		\hline     
		\multirow{4}{*}{$T/\infty/F$} & \multirow{4}{*}{ Multiple orders without outside source} & \cite{teunter2002inventory} \\
		& & \cite{pincce2011inventory} \\
		& & \cite{inderfurth2008decision} \\
		& & \cite{inderfurth2013advanced} \\
		\hline
		$T/1/F$ & Delayed single order without outside source & \cite{cattani2003good} \\
		\hline
		\multirow{4}{*}{$T/1/Z$} & \multirow{4}{*}{ Single order without outside source} & \mbox{\cite{behfard2018last}} \\
		& & \mbox{\cite{hur2018end}} \\
		& & \mbox{\cite{teunter1999end}} \\
		& & \mbox{\cite{jack2000optimal}} \\
		\hline
	\end{tabular}
	\caption{\label{Table: Benchmark Models and Related Study} Relation between the benchmark models and the previous studies. Notation for the models is presented in Table \ref{Table: Notation for Benchmark Models}.  }
\end{table}

\filbreak

\section{Structural Results} \label{Section: Structural Results} 
This section provides our structural results regarding the models and the analytical insights regarding the management of end-of-life problem. In Subsection \mbox{\ref{Subsection: Reformulation}}, we reduce the computation of DP algorithms by reformulating the most general problem $D/\infty/F$. Subsection \mbox{\ref{Subsection: Structural Results for D8F}} includes our main structural results for $D/\infty/F$. Finally, in Subsection \mbox{\ref{Subsection: Other Useful Structural Results}}, we show further analytical insights that motivate the use of an outside source and flexibility in placing orders.
\subsection{Structural Results for Dynamic Programming Computations} \label{Subsection: Reformulation}
We reduce the computation of the stopping cost $S(k,x_k)$ in $D/\infty/F$ and other benchmark models by reformulating $V^*(x)$ in \eqref{Equation: V*(x)}. To that end, we further define for each $k\in\TT$ and $x\in\ZZ_+$ the new cost of stopping by $\widetilde{S}(x):=c_{4}x$ and the new one-period operation cost by
\begin{align}\label{Equation: tilde(C)(k,x)}
\widetilde{C}(k,x) := & c_1 \EE\sqb{\int_k^{k+1} e^{-\delta(u-k)} \of{ x-(N_u-N_k)}^+ du  } + \EE\sqb{ \int_{k}^{k+1} e^{-\delta(u-k)} \sqb{c_2(u) -c_3(u)} \, dN_u } \notag \\
 & - \EE\sqb{\int_k^{(k+1)\wedge \sigma_x^k} e^{-\delta(u-k)} c_2(u) \, dN_u}.
\end{align}

Then, we use the following proposition to reformulate $V^*(x)$ in \eqref{Equation: V*(x)}. 
\begin{prop}
	\label{Prop: Reformulate Obj. Fnc.}
	Let $x\in\ZZ_+$ and define the new reformulated problem by
	\begin{align} \label{Equation: tilde(V)*(x)}
	\widetilde{V}^*(x) :=  \inf_{\pi\in\Pi} \E\sqb{ \condition{ \sum_{k=0}^{\tau-1} e^{-\delta k} \bigg(c(\mu_k(X_k)) + \widetilde{C}(k,X_k+\mu_k(X_k))\bigg)  + e^{-\delta\tau} \widetilde{S}(X_\tau) } X_0 = x }
	\end{align}
	Then, the new problem is equivalent to the original one in the sense that $V^*(x) = \widetilde{V}^*(x)+A$ for every $x\in\ZZ_+$, where $A:= \E\sqb{\int_{0}^{T} e^{-\delta u} c_{3}(u) dN_u }\in\RR_+$ is a constant.
\end{prop}

The proof of Proposition \ref{Prop: Reformulate Obj. Fnc.} is in Appendix \ref{Section: Proofs of Subsection Reformulation}.

Recall that in the original problem $V^*(x)$, the stopping cost $S(k,x)$ includes an integral term and depends on both time $k$ and inventory $x$. In the equivalent problem $\tV^*(x)$, however, the stopping cost $\tS(x)$ depends only on inventory $x$ and so the computation of $\tV^*(x)$ takes less time. Hence, in view of Proposition \mbox{\ref{Prop: Reformulate Obj. Fnc.}}, we solve $\tV^*(x)$. Define the backward DP for each $k\in\cb{T-1,T-2,\dots,0}$ and $x_k\in\ZZ_+$ by
\begin{align}
\widetilde{V}(k,x_k) = & \min\cb{\tS(x_k), \tJ(k,x_k) }, \label{Equation: tilde(V)(k,x_k) DP Definition} \\
\tJ(k,x):= & \inf_{m\in\ZZ_+} \of{c(m) + \tG(k,x_k+m)}, \label{Equation: tilde(J)(k,x)} \\
\widetilde{G}(k,y) := & \widetilde{C}(k,y) + e^{-\delta} \EE\sqb{ \widetilde{V}\big(k+1, (y-(N_{k+1}-N_k))^+\big) }. \notag
\end{align}
Here, $\tG$ denotes the continuation cost, $\tJ$ denotes the continuation cost after finding the best order amount, and $\tV$ compares the stopping and the continuation costs. Also define the terminal condition by $\widetilde{V}(T,x_T) = \widetilde{S}(x_T) = c_{4}x_T$ for each $x_T\in\ZZ_+.$ Then, the next corollary states that the DP algorithm described by the recursion in \eqref{Equation: tilde(V)(k,x_k) DP Definition} can be used to solve this new re-formulated problem $\tV^*(x)$. 

\begin{cor}\label{Cor: DP solves tilde(V)*}
  \mbox{$\widetilde{V}(0,x)=\widetilde{V}^*(x)$} for every $x\in\ZZ_+$. Moreover, the optimal order amount $\widetilde{\mu}^*_k(x_k)$ attains the infimum in \eqref{Equation: tilde(J)(k,x)}. Furthermore, an optimal stopping time $\widetilde{\tau}^*\in\cT$ is the one that stops the process if $\widetilde{S}(x_k)\leq \tJ(k,x_k)$ in \eqref{Equation: tilde(V)(k,x_k) DP Definition}, in other words, $\widetilde{\tau}^*:= \min\{ k\in\TT \,:\, \tS(X^{\tpi^*}_k)\leq \tJ(k,X^{\tpi^*}_k) \}$, where $\tpi^*:=(\widetilde{\tau}^*, \widetilde{\mu}_0^*,\dots, \widetilde{\mu}_T^*)$ is an optimal policy.
\end{cor}

Finally, the following proposition enables us to compute $\widetilde{C}$ so that we can compute $\widetilde{V}(0,x)$. 
\begin{prop}
	\label{Proposition: Converting tilde(C)}
	For every $k\in\TT$, the one-period operation cost $\widetilde{C}(k,x)$ in \eqref{Equation: tilde(C)(k,x)} satisfies
	\[
	\widetilde{C}(k,0) = \int_{k}^{k+1} e^{-\delta(u-k)} [c_2(u)-c_{3}(u)]  \lambda(u)\,du 
	\]
	when $x=0$, and
	\begin{align}\label{Equation: tilde(C)(k,x) deterministic}
	\widetilde{C}(k,x)
	= & c_{1} \sum_{n=0}^{x-1} \sum_{i=0}^{n} \int_{k}^{k+1} e^{-\delta(u-k)} e^{-(\Lambda(u)-\Lambda(k))} \frac{(\Lambda(u)-\Lambda(k))^i}{i!} \,du \notag \\
	& + \int_{k}^{k+1} e^{-\delta (u-k)} [c_2(u)-c_{3}(u)] \lambda(u)\,du \notag \\
	& - \sum_{i=0}^{x} \int_{k}^{k+1} e^{-\delta(u-k)} c_2(u) \lambda(u) e^{-\Lambda(u)-\Lambda(k)}\frac{(\Lambda(u)-\Lambda(k))^i}{i!} \,du
	\end{align}
	when $x\geq 1$.
\end{prop}

The proof of Proposition \ref{Proposition: Converting tilde(C)} is in Appendix \ref{Section: Proofs of Subsection Reformulation}.

\subsection{Structural Results for D/$\boldsymbol{\infty}$/F}
\label{Subsection: Structural Results for D8F}

\subsubsection{Characterization of Optimal Solution}

The main problem $D/\infty/F$ falls into the category of an optimal stopping problem with additional decisions. More specifically, the inventory situation considered is a multi-period problem with lost-sales in a setting with periodic review, finite horizon, non-stationary stochastic demand, non-stationary costs (continuously charged), and zero ordering lead time. We first present a structural result for our most general case and then discuss special cases that lead to more structural properties.  For each time $k\in\TT$, define 
\begin{align*}
R^S_k :=& \{x\in\ZZ_+: \tS(x) \leq \tJ(k,x) \},\\
R^O_k :=& \Big\{x\in\ZZ_+ : \inf_{m\geq 1} (c(m)+\tG(k,x+m)) < \min\{ \tS(x), \tG(k,x)\}\Big\},\\
R^C_k := & \{x\in\ZZ_+ : \tV(k,x) = \tG(k,x) < \tS(x) \}.
\end{align*}
Here, the sets $R^S_k, R^O_k$ and $R^C_k$ are respectively called stopping set, ordering set and continuation set. They denote the set of inventory levels that we stop, order and continue, respectively.
\begin{prop}
	\label{Proposition: Solution to D8F}
	The optimal solution $\widetilde{\pi}^*$ defined in Corollary \ref{Cor: DP solves tilde(V)*} can be summarized by three disjoint regions for the incoming stock value for every decision epoch as well as order-up-to levels. 
	\begin{enumerate}[leftmargin=*]
		\item The ordering region $\{(k,x)\in\TT\times\ZZ_+ : x\in R^O_k\}$, the stopping region $\{(k,x)\in\TT\times\ZZ_+: x\in R^S_k\}$, and the continuation region $\{(k,x)\in\TT\times\ZZ_+ : x\in R^C_k\}$ are disjoint subsets of $\TT\times\ZZ_+$. 
		\item $\tpi^*$ is a policy such that if $X^{\tpi^*}_k\in R^S_k$, then we stop; if $X^{\tpi^*}_k\in R^C_k$, then we continue without taking an action; if $X^{\tpi^*}_k\in R^O_k$, then we place an order and increase inventory to an order-up-to level. Hence, for each inventory level in $R^O_k$, there exists an optimal order-up-to level. 
	\end{enumerate}
\end{prop}
The proof of Proposition \ref{Proposition: Solution to D8F} is in Appendix \ref{Section: Proofs of Subsection Structural Results for D8F}.\\
Proposition \mbox{\ref{Proposition: Solution to D8F}} states an intuitive result that, among all different order amounts and stopping times, we can describe an optimal solution by using disjoint ordering, continuation and stopping regions. Moreover, breaking the ties between costs does not cause any problem so that the regions can be disjoint.

Next, we discuss two special cases for the environment described by Proposition \mbox{\ref{Proposition: Solution to D8F}}. The first special case of our model can be obtained by removing the possibility to stop, or in other words, by requiring the manufacturer to continue until the end-of-horizon. In such a case, a time-dependent $(s,S)$ policy is optimal to characterize the ordering decisions. We refer the reader to \citet[Theorem 4.2]{beyer2010markovian} for the optimality of $(s,S)$ policy in a more general structure, except that we compute costs in continuous time. Still, \citet[Theorem 4.2]{beyer2010markovian} applies here since \cite{frenk2019exact} shows that the one-period operation cost $C$ is convex.

The second special case of our model is obtained by setting the setup cost as $K=0$, allowing backordering, and restricting the penalty cost to take place at review periods only (rather than using a stopping time $\sigma_x$). In such a case, an optimal stopping policy is the two-sided threshold policy that stops the process at time $t$ if the inventory level $x_t\notin[\underbar{x}_t,\bar{\mathrm{x}}_t]$, where $\underbar{x}_t, \bar{\mathrm{x}}_t\in\RR$ are constants. We refer the reader to \citet[Proposition 12]{oh2016characterizing} for the proof. Under our general framework, however, we do not observe such thresholds in the numerical analysis. For instance, when the inventory level drops to zero, the manufacturer may wish to place an order rather than to stop, eliminating the possibility of a lower threshold.

Finally, it is worth mentioning that the non-stationary parameters and the existence of a positive setup cost limit the known analyses to be implemented. We illustrate these limitations by solving $D/\infty/F$ numerically with a data set presented by \cite{frenk2019exact} and $K=2000$. Table \ref{Table: Subset of Regions for Remark} demonstrates the ordering, stopping and continuation regions. We can see that the ordering and stopping regions are intertwined (e.g.,\@ $k=23$). Therefore, even if there exists a threshold re-order level, it is possible that we stop when inventory is below the threshold, rather than placing an order. Moreover, the regions are neither convex nor do they have monotone boundaries. The stopping region does not have monotone boundary possibly because of the decreasing intensity rate and the cost of outside source. On the one hand, if the intensity rate declines over time, then the manufacturer may not take a preventive stopping action (which may prevent lost-sales) although the inventory is small (e.g.,\@ $x\in\{3,4,5\}$). This makes the stopping region smaller in the last stages. On the other hand, when the cost of outside source declines over time, the stopping decision becomes more attractive, making the stopping region larger. Hence, the stopping region may be affected oppositely by the intensity rate and the cost of outside source.

\begin{table}[h!]
	\centering
	\begin{small}
	\renewcommand{\arraystretch}{0.9}
	\begin{tabular}{|c|ccccccccccc|}
		\hline
		\diagbox[width=3em]{$x$}{$k$} & 0                    & 18                    & 22                    & 23                    & 24                    & 25                    & 40                    & 43                    & 44                    & 48                    & 49                    \\ \hline
		\multicolumn{1}{|l|}{846} & \multicolumn{1}{l}{} & \multicolumn{1}{l}{}  & \multicolumn{1}{l}{}  & \multicolumn{1}{l}{}  & \multicolumn{1}{l}{}  & \multicolumn{1}{l|}{} & \multicolumn{1}{l}{}  & \multicolumn{1}{l}{}  & \multicolumn{1}{l}{}  & \multicolumn{1}{l}{}  & \multicolumn{1}{l|}{} \\ \cline{8-8}
		770                       &                      &                       &                       &                       &                       &                       & \multicolumn{1}{c|}{} &                       & \multicolumn{2}{c}{Stop}                      &                       \\ \cline{9-9}
		375                       &                      &                       &                       &                       &                       &                       &                       & \multicolumn{1}{c|}{} &                       &                       &                       \\ \cline{10-10}
		336                       &                      &                       &                       &                       &                       &                       &                       &                       & \multicolumn{1}{c|}{} &                       &                       \\ \cline{11-11}
		327                       &                      &                       &                       &                       & \multicolumn{3}{c}{Continue}                                          &                       &                       & \multicolumn{1}{c|}{} &                       \\ \cline{2-4} \cline{12-12} 
		7                         &                      &                       & \multicolumn{1}{c|}{} &                       &                       &                       &                       &                       &                       &                       &                       \\
		6                         &                      &                       & \multicolumn{1}{c|}{} &                       &                       &                       &                       &                       &                       &                       &                       \\ \cline{5-5} \cline{8-9}
		5                         &                      &                       &                       & \multicolumn{1}{c|}{} &                       & \multicolumn{1}{c|}{} &                       & \multicolumn{1}{c|}{} &                       &                       &                       \\ \cline{7-7}
		4                         &                      & \multicolumn{2}{c}{Order}                     & \multicolumn{1}{c|}{} & \multicolumn{1}{c|}{} &                       &                       & \multicolumn{1}{c|}{} &                       &                       &                       \\ \cline{6-6} \cline{12-12} 
		3                         &                      &                       &                       & \multicolumn{1}{c|}{} &                       &                       &                       & \multicolumn{1}{c|}{} &                       & \multicolumn{1}{c|}{} &                       \\ \cline{5-5} \cline{10-11}
		2                         &                      &                       & \multicolumn{1}{c|}{} &                       &                       &                       &                       &                       &                       &                       &                       \\
		1                         &                      &                       & \multicolumn{1}{c|}{} &                       &                       & \multicolumn{3}{c}{Stop}                                              &                       &                       &                       \\ \cline{4-4}
		0                         &                      & \multicolumn{1}{c|}{} &                       &                       &                       &                       &                       &                       &                       &                       &                       \\ \hline
	\end{tabular}
	\end{small}
	\caption{An illustrative optimal solution and the corresponding regions. Horizontal axis represents time $k$ and vertical axis represent inventory level $x$.}
	\label{Table: Subset of Regions for Remark}
\end{table}

\FloatBarrier

\subsection{Structural Results Regarding Stopping Times} \label{Subsection: Other Useful Structural Results}
Stopping time plays a crucial role in the end-of-life inventory management problem. With the scope defined in this paper, stopping time indicates the time after which all the subsequent demand is to be satisfied by an alternative source. The alternative source can be any origin (own, any supplier, a competitor, a remanufacturing repair shop, etc. -- we simply call it the outside source) which agreed to cover the demand for the part considered until the end of horizon. Hence, it is important for the end-of-life management executives to warn the outside source on the possibilities of the time to switch  - a crucial information for the operation of the outside source. Hence, in this section we aim to obtain some structural results on stopping times so as to smooth out the problems that may occur during the realization of this transition. 

The results we have in the following subsections are valid for any $t$, where $t$ represents the time at which we review the inventory system. The only information we need to know is the state information (time $t$ and inventory $x$) for the details to work on. Finally, we note that even if we allow $t$ to be an element of $[0,T]$, we present all the results for $t=0$ to avoid additional notation.

\subsubsection{Bounds on the Switching Time}\label{Subsubsection: Bound on the Switching Time}
We consider the manufacturer at a later stage in the end-of-life phase (any time $t$) when existing inventory looks sufficient, and no further order is expected to be placed. We aim to extract bounds on the time to stop when the manufacturer is willing to set an exact time (which is called the switching time) and communicate it with the provider of outside source. 

We consider $t=0$ without loss of generality. Let $\cC(x,\tau)$ denote the total cost incurred if the inventory is $x$ and if we decide to stop at $\tau\in \TT$. Then, the function $\cC:\ZZ_+\times\TT\rightarrow \RR_+$ is defined by
\begin{align*}
\cC(x,\tau):=&\EE\sqb{ \sum_{k=0}^{\tau-1} e^{-\delta k} C(k,X_k)+ e^{-\delta \tau} S(\tau,X_\tau) }
\end{align*}
which is equivalent to
\begin{align} \label{Equation: cC(x,tau) Defn}
\cC(x,\tau) = & c_{1} \EE\sqb{\int_{0}^{\tau} e^{-\delta u} (x-N_u)^+ du }
+ \EE \sqb{ \int_{\tau \wedge \sigma_x}^{\tau} e^{-\delta u} c_{2} (u) dN_u }  \notag \\
&+ \EE\sqb{\int_{\tau}^{T} e^{-\delta u} c_{3}(u) dN_u }
+ c_{4} e^{-\delta \tau} \EE[(x- N_\tau)^+]. 
\end{align}

While offering the next insights, we allow $\tau$ to be continuous in order to take derivative of $\cC$ with respect to $\tau$. Hence, we use the same form of $\cC(x,\tau)$ in \mbox{\eqref{Equation: cC(x,tau) Defn}} while allowing $\tau$ to be any value in $[0,T]$. In other words, we extend the definition of $\cC$ as function $\cC: \ZZ_+\times[0,T]\rightarrow\RR_+$ given by \mbox{\eqref{Equation: cC(x,tau) Defn}}. We derive the insights by finding upper and lower bounds on the solution of the following problem:
\begin{equation}\label{Equation: V^(S/1/Z) fixed x_0}
\inf_{\tau\in[0,T]} \cC(x,\tau).
\end{equation}
To analyze the properties of $\cC$, we make the same general assumptions with \mbox{\cite{frenk2019exact}}: \\
\hypertarget{NON-INC}{NON-INC:} The functions $c_3$ and $u\mapsto \widetilde{c}_2(u):=c_2(u)-c_3(u)$ are right-continuous with left-limits, piece-wise smooth (that is, differentiable except at finitely many points) and non-increasing. \\
\hypertarget{POS}{POS:} The functions $c_3$ and $\widetilde{c}_2$ are non-negative (maps from $[0,T]$ to $\RR_+$) and the quantities $c_4$ and $c_1-\delta c_4$ are in $\RR_+$. 

The assumption POS means that the parameters are positive and it is quite general as we are already minimizing the costs. On the other hand, NON-INC states that the functions are non-increasing. Such assumption is more relevant in the end-of-life phase. \mbox{The cost of outside source, $c_3$,} decreases over time, as the manufacturer is better prepared to use such a source. Moreover, the lost sales penalty, $\widetilde{c}_2$, decreases over time, as the manufacturer is better prepared for the possibility of an insufficient inventory while nearing the end of horizon.

Let $\tau^*\in[0,T]$ denote a best switching time that attains the infimum in \eqref{Equation: V^(S/1/Z) fixed x_0}. The following two propositions show upper and lower bounds on $\tau^*$. 

\begin{prop} \label{Proposition: Upper bound on tau^*}
	Assume NON-INC and POS. Let $\tau^{ub}$ be the smallest $\tau$ value satisfying
	\begin{equation} \label{Equation: Upper bound on tau^*}
	\tilde{c}_2(T) \geq \PP\cb{N_\tau \leq x-1} [c_{2}(\tau) + c_{4}  ].
	\end{equation}
	If the inequality does not hold for any $\tau\in[0,T]$, then let $\tau^{ub}=T$. Then, $\tau^* \leq \tau^{ub}$.
\end{prop}
The proof of Proposition \ref{Proposition: Upper bound on tau^*} is in Appendix \ref{Section: Proof of Results in Subsubsection: Inventory at 0 affects switching time}. 
\begin{prop}\label{Proposition: Lower bound on tau^*}
	Assume NON-INC and POS. Also assume that $\lambda$ is a non-increasing function. Let $\tau^{lb}$ be the largest $\tau$ value satisfying $\lambda(\tau) \geq 1$ and
	\begin{equation} \label{Equation: Lower bound on tau^*}
	\PP\cb{N_\tau\leq x-1} [c_{2}(\tau) + c_{4}  ] \geq x(c_{1}-\delta c_{4}) + \tilde{c}_2(0).
	\end{equation}
	If the inequality does not hold for any $\tau\in[0,T]$, then let $\tau^{lb}=0$. Then, $\tau^{lb}\leq \tau^*$.
\end{prop}
The proof of Proposition \ref{Proposition: Lower bound on tau^*} is in Appendix \ref{Section: Proof of Results in Subsubsection: Inventory at 0 affects switching time}. 

We extract the following insights for the bounds on the best switching time $\tau^*$. 
\begin{itemize}[leftmargin=*]
	\item The upper and lower bounds depend on the inventory level at the later stage, $x$, the demand until we stop, $N_\tau$, and the cost of outside source, $c_2(\tau)$.
	\item The upper bound $\tau^{ub}$ can be less than the remaining time in the horizon, $T$, if the number of arrivals up to $\tau^{ub}$ sufficiently exceeds the inventory level $x$ (so that the term $\PP\{N_{\tau^{\text{ub}}} \leq x-1\}$ is small) and if the outside source is cheap (so that the term $c_2(\tau^{\text{ub}})$ is small). Hence, an insufficient inventory level and a decreasing cost of outside source prompt the use of such a source.
	\item The observation for the lower bound complements the one for the upper bound. Indeed, $\tau^{lb}$ can be larger than $0$ if the number of arrivals before $\tau^{lb}$ is considerably less than inventory level $x$ (so that the term $\PP\{N_{\tau^{\text{lb}}} \leq x-1\}$ is small) and if the cost of outside source is still high (so that the term $c_2(\tau^{\text{lb}})$ is large). In such a case, it may be better to delay the use of the outside source.
\end{itemize}
Hence, the manufacturer, while being at a later stage in the phase, is more motivated to stop before the last period $T$ as long as the inventory level is sufficient to prevent ordering but not enough to cover the demand, and the outside source is relatively cheap. The properties above enable the manufacturer to communicate with the provider of the outside source.

\subsubsection{Stopping Times and Final Order Quantity}
\label{Subsubsection: S is increasing in tau}
Our next analysis reveals that when the system is at a later stage of the horizon, the decision on the time to stop affects the size of the final order. In turn, this creates a domino effect between order-up-to levels, where the previous levels are impacted by the subsequent levels. Thus, order quantities and stopping times are related. Such relation motivates the use of a model that incorporates both multiple orders and stopping time.

To materialize this relation and use the results for more information on the stopping times, we analyze the manufacturer's decisions at a later stage when a final order and a time to stop are to be decided together so that the outside source provider can be informed. Without loss of generality, we define $t=0$ as the current period, since one can update the non-stationary parameters to accommodate for a change in current time. Here, we aim to extract the relation between the switching time and the final order size. Thus, the objective is to choose an order amount $m$ and a switching time $\tau$ that solve
\begin{equation*} 
\inf_{\substack{\tau\in [0,T], \\ m\in\ZZ_+} } \of{ c(m) + \cC(m+x_0,\tau) }, \quad x_0\in\II,
\end{equation*}
where $c$ is the ordering cost function defined in \mbox{\eqref{Equation: Order cost function}} and $\cC$ is the combined operation cost function defined in \mbox{\eqref{Equation: cC(x,tau) Defn}}. By proceeding as in \cite{frenk2019exact}, we can show that $x\mapsto \cC(x,\tau)$ is a discrete-convex function for each fixed $\tau$, under the assumptions NON-INC and POS above. Then, we can follow the arguments presented by \cite{porteus2002foundations} and state that the $(s(\tau),S(\tau))$-policy is an optimal policy to describe the ordering decision, where the order-up-to level $S(\tau)$ and the re-order level $s(\tau)$ depend on the switching time $\tau$. Moreover, the following proposition shows that $\tau\mapsto S(\tau)$ is a non-decreasing function when suitable conditions hold. \vspace{2mm}

\begin{prop}\label{Proposition: S is increasing in tau}
	Assume NON-INC and POS. Also assume that $\lambda$ is a non-increasing function. Then, for every $\epsilon, \tau_1, \tau_2\in[0,T]$ such that (i) $\tau_1 < \tau_2$, (ii) $S(\tau_1) \geq S(\tau_2)$, (iii) $S(\tau_1) \leq \Lambda(\tau_2)$ and
	$$
	(iv)\,\, c_{1} \leq \sqb{\displaystyle\frac{\Lambda(u) - S(\tau_1)}{\Lambda(u)}} \lambda(u)c_{3}(u), \text{ for all } u\in[\tau_2,\tau_2+\epsilon],
	$$
	we have $S(\tau_2) \leq S(\tau_2+\epsilon)$. 
\end{prop} 

The proof of Proposition \ref{Proposition: S is increasing in tau} is in Appendix \ref{Section: Proof of Results in Subsubsection S is increasing in tau}. \\
Proposition \ref{Proposition: S is increasing in tau} states that small perturbations in switching time (from $\tau_1$ to $\tau_2$) may not affect the order-up-to level; however, moderate deviations (from $\tau_1$ to $\tau_2+\epsilon$) raise the level. If the manufacturer decides to stop at a later time without updating the level (conditions (\textit{i}) and (\textit{ii})), but expected total demand exceeds the level (condition (\textit{iii})), and the cost rate of outside source, $\lambda(u)c_3(u)$, is still high in the infinitesimal future (condition ($iv$)), then the order-up-to level increases. Proposition \ref{Proposition: S is increasing in tau} enables us to extract the following insights.
\begin{itemize}[leftmargin=*]
	\item The time when the manufacturer decides to stop has an impact on the previous order amounts.
	\item If the disposal time of the inventory is delayed moderately (in the sense of conditions ($i$)-($iv$)), then the final order amount that will be used to satisfy the demand increases.
	\item In case the outside source is not available before some time (say, $\tau_2$), then most likely we need to increase the order-up-to level.
\end{itemize} 
This analysis reveals that stopping time and order levels are alternatives and complements for managing the end-of-life inventory system. The use of both strategies simultaneously is likely to fine-tune the results, yielding less expected costs.

\subsubsection{Distribution of Optimal Stopping Time}\label{Subsubsection: Distribution of Stopping Time}
In the previous two subsections, we consider the time periods after which it is unlikely to place an order, and analyze possible properties of the switching time. We consider those structural results as important since they can be utilized to support decision-making in the later stages of the end-of-life problem. Here we go one-step further. Once the optimal solution strategy to the DP problem is obtained, we can trace the solution in any period, as explained by Proposition \ref{Proposition: Solution to D8F}. It turns out that one can do better: using the backwards trace, it is possible to obtain the distribution of the stopping time that is dictated by the optimal solution of DP.

To be more specific, optimal stopping decisions are functions of inventory level at designated times. Hence, the randomness of the stopping times is solely dependent on the demand process and optimal strategy implemented, where the latter can be summarized by when-to-order and how-much-to-order decisions together with stopping decisions. Given the discrete nature of time, it is possible to compute the distribution of stopping times. In Appendix \mbox{\ref{Sec: Computation of Distribution of stopping time}}, we show that we can compute the probabilities of the stopping time distribution efficiently by using the stated recursions. One can use the structure of the recursive relations to come up with the distribution of stopping times at any point in time given the inventory level. Of course, this might require an expensive operation; however, depending on the characteristic and value of the inventory carried, it might be reasonable to put the effort. \mbox{\cite{hur2018end}} describes such an environment where all the effort is spent to estimate the demand distribution for the whole horizon to implement an order policy. Thus, the comments made at the beginning of Section \mbox{\ref{Section: Structural Results}} regarding use of stopping times can be realized with the knowledge of the relevant probability distribution at any time.

\section{Numerical Analysis}
\label{Section: Numerical Analysis}

In Section \ref{Section: Structural Results} we demonstrate with some structural results that the flexibilities that can be considered within the end-of-life inventory management problem are promising. This section provides numerical results regarding the output of the dynamic programming algorithms and the analytical results presented earlier. For numerical calculations, we further assume that the cost of outside source is given as $c_{3}(u)=\bar{c}_{3} e^{-\gamma u}$, where $\bar{c}_{3}\in\RR_+$ is a constant and $\gamma\in\RR_+$ is the decline rate. Moreover, we assume that $c_2(u)=\bar{c}_2+c_{3}(u)$ for some $\bar{c}_2\in\RR_+$ which we interpret as the penalty of lost sales. We allow $c_4$ to be negative or positive. Finally, to facilitate computations, we assume that the intensity function $\lambda$ is piecewise constant whose value changes at every $t\in\cb{0,1,\dots,T-1}$ and it is constant during $[t,t+1)$. Table \ref{Table: Our numerical Values} shows the set of parameter values and Table \ref{Table: Intensity Functions} shows the set of intensity functions used in the numerical analysis. We note that the ranges as stated in Table \ref{Table: Our numerical Values} incorporate in relative terms the case data considered by \cite{frenk2019exact}. Additionally, we specify a fixed ordering cost and various forms of the intensity function to represent the rate of decrease, as presented in Table \ref{Table: Intensity Functions}. Note that for different cases, the demands are all comparable as the total expected demand over the horizon is kept constant. The parameter settings and their corresponding numbers are presented in Appendix \ref{Section: Number of Parameter Settings}.

We code our models by using MATLAB and run them on a laptop computer with an Intel(R) Core(TM) i7-7700HQ processor with 2.80GHz CPU. The computation of one-period operation cost $C(k,x)$ (for all $k$ and $x$ values) takes approximately 800 and 1750 seconds of CPU time when $T=50$ and $T=100$, respectively. The computation of DP algorithm takes approximately 25 and 60 seconds of CPU time when $T=50$ and $T=100$, respectively. We verified our code by comparing the output of benchmark models and the output presented by \cite{frenk2019exact}. 

\begin{table}[h!]
	\centering
	\begin{tabular}{|l|l|}
		\hline
		Parameter Name & Set of Values \\
		\hline 
		Unit Procurement Cost & $\bar{c}=100$ \\
		Setup Cost & $K\in\cb{0, 10\bar{c}, 50\bar{c} }$ \\
		Holding Cost & $c_{1}=0.01\bar{c}$ \\
		Penalty Cost & $\bar{c}_2 \in \cb{2\bar{c}, 10\bar{c} }$ \\ 
		$c_{3}(0)$ at time zero & $\bar{c}_3=2\bar{c}$  \\
		Discount of $c_{3}$ & $\gamma\in\cb{10^{-6}, 0.01}$ \\
		Scrapping Cost & $c_{4}\in\cb{\bar{c}/4, -\bar{c}/4 }$ \\ 
		Planning Horizon & $T\in\cb{50,100}$ \\ 
		Time Discount &$\delta\in\cb{10^{-6},0.005}$ \\
		Expected Total Demand & $\int_{0}^{T}\lambda(u)\,du=500$ \\
		Intensity Functions & Convex, Concave, Linear, Constant \\
		& Presented in Table \ref{Table: Intensity Functions}  \\
		\hline
	\end{tabular}
	\caption{\label{Table: Our numerical Values} Parameter values used in numerical analysis. The total number of parameter settings is 384. Base case parameters are $\bar{c}_2=2\bar{c}, \gamma=0.01, c_4 = \bar{c}/4, T=50, \delta=0.005$, with convex $\lambda$. } 
\end{table}

\begin{table}[h!]
	\centering
	\def\arraystretch{1.2}
	\begin{tabular}{|l|c|c|}
		\hline
		Function Type & \multicolumn{2}{c|}{Expression} \\
		\hline 
		& $T=50$ & $T=100$ \\ \cline{2-3} 
		Convex & $\lambda(t)=\lambda_0(0.9)^t$ & $\lambda(t)=\lambda_0(0.96)^t$ \\
		Concave & $\lambda(t)=\lambda_0 - (0.045 t)^3$ & $\lambda(t)=\lambda_0-(0.015 t)^3$ \\
		Linear & $\lambda(t)=\lambda_0-0.392 t$ & $\lambda(t)=\lambda_0-0.099 t$ \\
		Constant & $\lambda(t)=\lambda_0$ & $\lambda(t)=\lambda_0$ \\
		\hline
	\end{tabular}
	\caption{  Piecewise constant intensity functions $\lambda$ used in numerical analysis. The value of $\lambda(t)$ changes at every $t\in\cb{0,1,\dots,T-1}$ and it is constant during $[t,t+1)$. The initial point $\lambda_0\in\RR_+$ is selected such that expected total demand $\int_0^T \lambda(t)\,dt$ is equal to $500$.}
	\label{Table: Intensity Functions}
\end{table}

\FloatBarrier

We present our computational results in two subsections. In Subsection \ref{Subsection: Analysis of the Benefits of the Proposed Approach}, we compare the benefits of our approach with the benchmark models. In Subsection \ref{Subsection: Sensitivitiy Analysis}, we analyze the effects of problem parameters. These analyses give us further insights on how our approach can be used to handle the end-of-life management problem effectively.

\subsection{Analysis of the Benefits of the Proposed Approach }\label{Subsection: Analysis of the Benefits of the Proposed Approach}

In this subsection, we compare all models by using the parameter values in Table \ref{Table: Our numerical Values}. To be more specific, we consider the settings (as numbered) which are presented in Appendix \ref{Section: Number of Parameter Settings}. As comparison basis, we consider the percentage increase in the expected discounted total cost over the horizon for not employing a model which utilizes more flexibility (or flexibilities) over the assumed current model with those flexibilities. We summarize the benefits in four subsections: The first three subsections assess the contribution of any specific flexibility over the current one in an isolated manner, whereas the fourth subsection analyzes combination effects. 

We initially demonstrate the benefits over a base case with the parameters $\bar{c}_2=2\bar{c}=200, \gamma=0.01, c_4 = \bar{c}/4=25, T=50, \delta=0.005$, convex $\lambda$. The setting number for base case is 1 in Appendix \ref{Section: Number of Parameter Settings}. The base case is a setting which is one of the closest to the parameters used in Frenk et al (2019a). We then show results for all the settings and report minimum, average, and maximum percentage increase in the expected discounted cost if those flexibilities are not considered. We end each subsection with a remark summarizing the findings.

\subsubsection{Loss for Not Allowing Multiple Orders}

We first present the comparisons under the base case. Table \ref{Table: T1ZvsT8F} presents percent loss if the number of orders is limited to 1 compared to the possibility of multiple ordering under two cases: with no stopping time ($a=T$ as given in the taxonomy) and with stopping time ($a=D$). The case with stopping time is presented in parentheses for various values of $x$ and $K$.

\begin{table}[h!]
	\centering
	\begin{tabular}{|c|ccc|}
		\hline
		\diagbox[width=3em]{$K$}{$x$}  & 0     & 100   & 250 \\
		\hline
		0     & 17.2 (12.1)  & 21.4 (15.1)  & 31.1 (21.1) \\
		1000  & 6.3 (2.8)  & 8.6 (4.4)  & 15.3 (8.9) \\
		5000  & 0.5 (0.0)  & 2.6 (1.9)  & 9.3 (7.2) \\
		\hline
	\end{tabular}%
	\caption{\label{Table: T1ZvsT8F} Percentage difference $100\times\frac{T/1/Z - T/\infty/F}{T/\infty/F}$ for different initial inventory $x$ and setup cost $K$ values. The numbers in parentheses show the value of $100\times\frac{D/1/Z - D/\infty/F}{D/\infty/F}$ under the same parameter setting. Base case parameters are used. }
\end{table}%

\FloatBarrier

As expected, the use of stopping time is an effective tool as observed with lower percentages in parentheses. As expected, when $x=0$ and $K=5000$, we have small percentages indicating that the traditional approach, assuming a large fixed ordering cost and hence ordering only once, has a strong logic. However, with some initial inventory, one can observe that the penalty of not employing a more flexible approach can be significant. 

We present our results for all our runs in Table \ref{Table: D1ZvsD8F}. Note that we only report the expected percent loss figures under the case where we employ stopping time. For the moderate value of the fixed ordering cost ($K=1000$), the average penalty percentages for different initial inventory values are all above 10\%. When we analyze the settings where we attain maximum or minimum values, we notice that most are the settings where we assume a constant intensity rate over time (\#125, \#113, \#121). This is expected as the use of flexibilities under some settings are more pronounced when the system is almost stationary or have no effect for the remaining settings.

\begin{table}[h!]
	\centering
	\begin{tabular}{|c|ccc|ccc|ccc|}
		\hline
	 	\diagbox[width=3em]{$K$}{$x$}  &     & 0    &          &     & 100  &          &     & 250  &          \\ \hline
	 	&     &  \%  & Set\#    &     &   \% &  Set\#   &     &  \%  &  Set\# \\
		& Max & 60.4 & 125 & Max & 73.4 & 125 & Max & 70.5 & 62  \\
		0     & Avg & 24.7 &          & Avg & 29.7 &          & Avg & 32.3 &          \\
		& Min & 9.0  & 11  & Min & 11.2 & 11  & Min & 3.1  & 121 \\ \hline
		& Max & 31.9 & 125 & Max & 43.3 & 125 & Max & 45.3 & 62  \\
		1000  & Avg & 10.0 &          & Avg & 14.2 &          & Avg & 17.7 &          \\
		& Min & 1.6  & 11  & Min & 2.6  & 12  & Min & 0.0  & 121 \\ \hline
		& Max & 11.1 & 125 & Max & 22.3 & 125 & Max & 26.9 & 109 \\
		5000  & Avg & 1.9  &          & Avg & 6.2  &          & Avg & 10.5 &          \\
		& Min & 0.0  & 27  & Min & 0.0  & 113 & Min & 0.0  & 49  \\ \hline
	\end{tabular}
	\caption{ \label{Table: D1ZvsD8F} Percentage difference $100\times\frac{D/1/Z - D/\infty/F}{D/\infty/F}$ for different initial inventory $x$ and setup cost $K$ values. For each $x$ and $K$, we present maximum, average and minimum values over all parameter settings. Set\# shows the setting numbers that attain the maximum or minimum. Setting numbers are presented in Appendix \ref{Section: Number of Parameter Settings}. } 
\end{table}
\FloatBarrier

\begin{rem}\label{Remark: Allowing multiple orders is important}
	Allowing multiple orders is important for systems with reasonable fixed ordering cost. However, the advantages may be offset by the use of stopping time and/or delaying the first order when there is some initial inventory.
\end{rem}

\subsubsection{Loss for Not Utilizing Stopping Time}
We present the comparisons under the base case parameters in two tables. Table \ref{Table: T1ZvsS1Z} presents the percent loss if only one order is given at time zero under two cases: with a switchover time ($a=S$ as given in the taxonomy) as to no stopping time ($a=T$) and with a stopping time ($a=D$) as to a switchover time ($a=S$). The latter case is presented in parentheses for various values of $x$ and $K$. Table \ref{Table: T8FvsD8F} presents the percent loss if multiple orders are allowed with a stopping time ($a=D$) as to no stopping time ($a=T$).

\begin{table}[h!]
	\centering
	\begin{tabular}{|c|ccc|}
		\hline
		\diagbox[width=3em]{$K$}{$x$}  & 0     & 100   & 250 \\
		\hline
		0     & 0.7 (4.2)   & 0.8 (5.2)  & 1.2 (7.8) \\
		1000  & 0.6 (4.1)  & 0.8 (5.0)  & 1.1 (7.5) \\
		5000  & 0.6 (3.9)  & 0.7 (4.6)  & 1.0 (6.6) \\
		\hline
	\end{tabular}%
	\caption{\label{Table: T1ZvsS1Z} Percentage difference $100\times\frac{T/1/Z - S/1/Z}{S/1/Z}$ for different initial inventory $x$ and setup cost $K$ values. The numbers in parentheses show the value of $100\times\frac{S/1/Z - D/1/Z}{D/1/Z}$ under the same parameter setting. Base case parameters are used. }
\end{table}%

\begin{table}[h!]
	\centering
	\begin{tabular}{|c|ccc|}
		\hline
		\diagbox[width=3em]{$K$}{$x$}  & 0     & 100   & 250 \\
		\hline
		0     & 0.4   & 0.5   & 0.7 \\
		1000  & 1.4   & 1.8   & 2.7 \\
		5000  & 3.9   & 4.6   & 5.6 \\
		\hline
	\end{tabular}%
	\caption{\label{Table: T8FvsD8F} Percentage difference $100\times\frac{T/\infty/F - D/\infty/F}{D/\infty/F}$ for different initial inventory $x$ and setup cost $K$ values. Base case parameters are used.}
\end{table}%

\FloatBarrier

Note that determining a switching time at the beginning does not constitute much improvement over no stopping time. However, moving to a stopping time improves the results, even under multiple orders. 

We present our results for all our runs in Table \ref{Table: S1ZvsD1Z}. Note that we only report the expected percent loss figures under the case where we employ stopping time as compared to switching time. The results show that the fixed ordering cost does not significantly affect the outcome as we limit ourselves to a single order. When we analyze the settings where we attain the least loss values, we notice that most are the settings where we assume a constant intensity rate over time (\#111). This is expected as we probably resort to stopping decision occasionally when we have almost stationary demand. Similarly, \#24 seems to be the setting where we attain maximum loss for not using stopping time. Setting \#24 is the case where demand is large at the beginning, the horizon is long and finally we gain a positive return (disposal value) when we stop with inventory, all indicating that a more precise selection of the disposal time increases the benefits.

\begin{table}[h!]
	\centering
	\begin{tabular}{|c|ccc|ccc|ccc|}
		\hline
		\diagbox[width=3em]{$K$}{$x$} &     & 0   &          &     & 100  &          &     & 250  &          \\ \hline
		&     & \%  & Set\#    &     & \%   & Set\#    &     & \%   &  Set\#  \\
		& Max & 8.9 & 24  & Max & 10.6 & 24  & Max & 14.8 & 24  \\
		0     & Avg & 4.6 &          & Avg & 5.5  &          & Avg & 8.0  &          \\
		& Min & 2.1 & 111 & Min & 2.4  & 111 & Min & 3.4  & 111 \\ \hline
		& Max & 8.8 & 24  & Max & 10.4 & 24  & Max & 14.4 & 24  \\
		1000  & Avg & 4.5 &          & Avg & 5.4  &          & Avg & 7.8  &          \\
		& Min & 2.0 & 111 & Min & 2.4  & 111 & Min & 3.3  & 111 \\ \hline
		& Max & 8.3 & 24  & Max & 9.7  & 24  & Max & 13.1 & 24  \\
		5000  & Avg & 4.2 &          & Avg & 4.9  &          & Avg & 7.0  &          \\
		& Min & 1.9 & 111 & Min & 2.0  & 89  & Min & 3.0  & 111 \\ \hline
	\end{tabular}
	
	\caption{\label{Table: S1ZvsD1Z} Percentage difference $100\times\frac{S/1/Z - D/1/Z}{D/1/Z}$ for different initial inventory $x$ and setup cost $K$ values. For each $x$ and $K$, we present maximum, average and minimum values over all parameter settings. Set\# shows the setting numbers that attain the maximum or minimum. Setting numbers are presented in Appendix \ref{Section: Number of Parameter Settings}. }
\end{table}

\FloatBarrier

\begin{rem}\label{Remark: Stopping time is important}
	Disposing the available inventory seems to be a critical decision, especially for some remaining time and inventory level combinations (see Subsection \ref{Subsection: Other Useful Structural Results} for supporting analytical results). Moreover, the dynamic selection of this time (via stopping time) as compared to determining at the beginning can be valuable in case the manufacturer has such flexibility. However, we notice that under the case where we allow for multiple orders, the effect of stopping time is reduced though not diminished. 
\end{rem}

\subsubsection{Loss for Not Delaying the First Order}
We first present the comparisons under the base case. Table \ref{Table: T1ZvsT1F} presents the percent loss if we allow for only one order but may delay the time to order under two cases: with a switchover time ($a=S$ as given in the taxonomy) and with a stopping time ($a=D$). The case with stopping time is presented in parentheses for various values of $x$ and $K$. 

\begin{table}[h!]
	\centering
	\begin{tabular}{|c|ccc|}
		\hline
		\diagbox[width=3em]{$K$}{$x$}  & 0     & 100   & 250 \\
		\hline
		0     & 0.0 (0.0)   & 2.3 (2.0)  & 10.1 (7.9) \\
		1000  & 0.0 (0.0)  & 2.3 (2.0)  & 9.9 (7.8) \\
		5000  & 0.0 (0.0)  & 2.2 (1.9)  & 9.1 (7.2) \\
		\hline
	\end{tabular}%
	\caption{\label{Table: T1ZvsT1F} Percentage difference $100\times\frac{T/1/Z - T/1/F}{T/1/F}$ for different initial inventory $x$ and setup cost $K$ values. The numbers in parentheses show the value of $100\times\frac{D/1/Z - D/1/F}{D/1/F}$ under the same parameter setting. Base case parameters are used. }
\end{table}%

\FloatBarrier

As we expect, there is no difference when $x=0$. Also, the results show that the fixed ordering cost does not significantly affect the outcome as we limit ourselves to a single order. As $x$ grows, we observe an increase in the losses. Note that, stopping time is a powerful tool as it partially compensates the mistake in the timing of the first order, and hence the benefits we observe in parentheses are smaller. 

We present our results for all our runs in Table \ref{Table: D1ZvsD1F}. Note that we only report the expected percent loss figures under the case where we also utilize stopping time. The observations made for Table \ref{Table: T1ZvsT1F} are valid here, as well. However, when looking at the percentages, the maximum values here are significant. When we analyze the settings where we attain the maximum loss values, we notice that most are the settings where we assume a constant intensity rate over time (\#125, \#109). This is expected as the use of flexibilities under some settings are more pronounced when the system is almost stationary.

\begin{table}[h!]
	\centering
	\begin{tabular}{|c|ccc|ccc|ccc|}
		\hline
		\diagbox[width=3em]{$K$}{$x$} &     & 0   &        &     & 100  &          &     & 250      &          \\ \hline
		&     & \%  & Set\#  &     & \%   & Set\#    &     & \%       & Set\# \\
		& Max & 0.0 & 1 & Max & 13.5 & 125 & Max & 32.3     & 125 \\
		0     & Avg & 0.0 &        & Avg & 6.2  &          & Avg & 14.8     &          \\
		& Min & 0.0 & 1 & Min & 1.2  & 12  & Min & 1.3      & 121 \\ \hline
		& Max & 0.0 & 1 & Max & 13.4 & 125 & Max & 31.3     & 61       \\
		1000  & Avg & 0.0 &        & Avg & 6.1  &          & Avg & 13.9     &          \\
		& Min & 0.0 & 1 & Min & 1.1  & 12  & Min & 0.0      & 121 \\ \hline
		& Max & 0.0 & 1 & Max & 13.2 & 125 & Max & 26.9 & 109 \\
		5000  & Avg & 0.0 &        & Avg & 5.2  &          & Avg & 10.4 &          \\
		& Min & 0.0 & 1 & Min & 0.0  & 113 & Min & 0        & 49  \\ \hline
	\end{tabular}
	
	\caption{\label{Table: D1ZvsD1F} Percentage difference $100\times\frac{D/1/Z - D/1/F}{D/1/F}$ for different initial inventory $x$ and setup cost $K$ values. For each $x$ and $K$, we present maximum, average and minimum values over all parameter settings. Set\# shows the setting numbers that attain the maximum or minimum. Setting numbers are presented in Appendix \ref{Section: Number of Parameter Settings}.  }
\end{table}

\begin{rem}\label{Remark: Delaying order is important.}
	In case the manufacturer is given the opportunity to order at any time, the cutoff initial inventory level which prevents ordering at time zero can be quite low. Therefore, the prevalent assumption that a final order is to be placed at time zero can be a strong assumption, possibly leading to significant losses. 
\end{rem}

\FloatBarrier

\subsubsection{Value of Combining the Features}

When we combine all the effects, the overall results indicate promising savings. Here we only present the comparisons under the base case. 

Table \ref{Table: D1ZvsD8F_2} presents the cases where we implement the optimal stopping strategy and record the percent loss if we do not use the flexibility of delaying the first order as well as multiple order opportunities. As noted before, the optimal stopping time can compensate the advantages when using other flexibilities. Nevertheless, we have significant losses if we do not implement other flexibilities even in the case where we have small $K$ and small $x$ values. Of course, the advantages reduce (or disappear for $x=0$) with larger $K$ values.  

\begin{table}[h!]
	\centering
	\begin{tabular}{|c|ccc|}
		\hline
		\diagbox[width=3em]{$K$}{$x$} & 0     & 100   & 250 \\
		\hline
		0     & 12.1  & 15.1  & 21.1 \\
		1000  & 2.8   & 4.4   & 8.9 \\
		5000  & 0.0   & 1.9   & 7.2 \\
		\hline
	\end{tabular}%
	\caption{\label{Table: D1ZvsD8F_2} Percentage difference $100\times\frac{D/1/Z - D/\infty/F}{D/\infty/F}$ for different initial inventory $x$ and setup cost $K$ values. Base case parameters are used.}
\end{table}

\FloatBarrier

Tables \ref{Table: T1ZvsD8F} and \ref{Table: T1ZvsD1F}, on the other hand, give us another interpretation. If we do not use our full flexibility scheme compared to the standard final order approach with no stopping, then our losses can be as much as 32\% when $x=250$ and $K=0$ in Table \ref{Table: T1ZvsD8F}. Almost half of the loss comes from not using the flexibility of ordering at any time even if we are going to order at most once; 17.7\% when $x=250$ and $K=0$ in Table \ref{Table: T1ZvsD1F}. Similar deductions can be made when comparing other cases. Of course, as $x$ gets smaller, we see the effect of delaying the order diminishing.

\begin{table}[h!]
	\centering
	\begin{tabular}{|c|ccc|}
		\hline
		\diagbox[width=3em]{$K$}{$x$}  & 0     & 100   & 250 \\
		\hline
		0     & 17.6  & 22.0  & 32.0 \\
		1000  & 7.7   & 10.6  & 18.4 \\
		5000  & 4.5   & 7.4   & 15.5 \\
		\hline
	\end{tabular}%
	\caption{\label{Table: T1ZvsD8F} Percentage difference $100\times\frac{T/1/Z - D/\infty/F}{D/\infty/F}$ for different initial inventory $x$ and setup cost $K$ values. Base case parameters are used.}
\end{table}%

\begin{table}[h!]
	\centering
	\begin{tabular}{|c|ccc|}
		\hline
		\diagbox[width=3em]{$K$}{$x$}  & 0     & 100   & 250 \\
		\hline
		0     & 4.9   & 8.1   & 17.7 \\
		1000  & 4.8   & 7.9   & 17.2 \\
		5000  & 4.5   & 7.4   & 15.5 \\
		\hline
	\end{tabular}%
	\caption{\label{Table: T1ZvsD1F} Percentage difference $100\times\frac{T/1/Z-D/1/F}{D/1/F}$ for different initial inventory $x$ and setup cost $K$ values. Base case parameters are used.}
\end{table}%

\FloatBarrier

\begin{rem}\label{Remark: Considering joint effects is important}
	Considering the joint effect of stopping time, order frequency and delaying the first order, we conclude that the fixed cost of ordering and the inventory level at time zero play an important role. If the initial inventory is large (in our numerical experiments, we take the largest $x$ value to be the half of the expected total horizon demand), then the management is advised to search for the feasibility of implementing a stopping time, as well as delaying the first order. However, if initial inventory level is small, then it is important to consider the possibility of implementing a stopping time and the multiple order option concurrently.
\end{rem} 

\subsection{Sensitivity Analysis and Managerial Insights} \label{Subsection: Sensitivitiy Analysis}

In this subsection, we analyze the performance of our proposed model in detail.  We consider pairwise comparisons of expected discounted total cost given different levels of a parameter utilized in the model.  We summarize these sensitivities in seven subsections: effect of demand structure, effect of outside source, effect of time horizon, effect of penalty cost, effect of time discount and effect of scrapping cost, and finally effect of incorrectly specifying the demand structure.

We only consider limited number settings to demonstrate the sensitivity.  We report all these sensitivities for different $K$ and $x$ values. We end each subsection with a managerial insight summarizing and generalizing the findings. The key approach while generating these insights is not necessarily related to the current decision framework only, but to support further decision-making needed to handle the complete end-of-life management problem as well. 

\subsubsection{Effect of Demand Structure} 
Table \ref{Table: V_convex vs. V_concave,T=50} shows the effect of demand structure on the model $D/\infty/F$ by showing the percent difference when intensity function $\lambda$ is concave and convex. If both the initial inventory and the setup cost are low ($x=0, K=0$), then the cost under convex demand is higher. The reason is that more demand is satisfied earlier when the intensity is convex. Therefore, the costs are discounted less. On the other hand, if the initial inventory is low yet the setup cost is high ($x=0, K=5000$), the cost under concave demand is higher, since more setup might be needed throughout the horizon under concave demand, as the decline rate of demand is lower. If we start with a very large initial inventory ($x=400$) implying that we may not need much ordering, then holding cost component dominates and hence we have a much higher cost for the concave case as inventory is depleted much slower.

Table \ref{Table: V_convex vs. V_concave, T=100} shows the results when $T=100$. The trend is similar to what is said for $T=50$. However, as the horizon is longer and total expected demand is the same for both time horizons, the expected drop in on-hand inventory for the convex case relative to concave is less, and hence percent differences for large $x$ values are not as large as the case $T=50$.

\noindent\textbf{Insight 1:} If we have either high initial inventory level $x$ or high setup cost $K$, it might be wise to encourage (even give incentives to) customers to come earlier -- hence make the demand rate look like convex compared to the original one. On the other hand, with small $x$ and $K$ combinations (northwest part of Table \ref{Table: V_convex vs. V_concave,T=50}), we may look for strategies making customers come later.

\begin{table}[h!]
	\centering
	\begin{tabular}{|c|ccccc|}
		\hline
		\diagbox[width=3em]{$K$}{$x$}  & 0      & 100    & 250    & 300   & 400    \\ \hline
		0    & -5.1 & -5.5 & -0.5 & 4.3 & 27.7 \\
		1000 & -1.5 & -2.9 & -0.3 & 3.1 & 21.0 \\
		5000 & 4.5  & 0.8  & 0.0  & 1.7 & 5.7  \\ \hline
	\end{tabular}
	\caption{\label{Table: V_convex vs. V_concave,T=50}  Percentage difference $100\times (V_{Concave} - V_{Convex}$) / $V_{Convex}$: Comparison of $\widetilde{V}(0,x)+A$ when demand is convex and concave (see Table \ref{Table: Intensity Functions} for definitions). The relevant parameters are $T=50, c_2 = 2\bar{c}, \gamma=0.01, c_4=c/4, \delta=0.005$. } 
\end{table}

\begin{table}[h!]
	\centering
	\begin{tabular}{|c|ccccc|}
		\hline
		\diagbox[width=3em]{$K$}{$x$} & 0       & 100     & 250     & 300     & 400    \\ \hline
		0    & -11.8 & -12.8 & -3.1  & 5.7   & 39.6 \\
		1000 & -11.1 & -14.4 & -11.2 & -5.6  & 28.1 \\
		5000 & -11.3 & -17.6 & -22.4 & -15.8 & 27.0 \\ \hline
	\end{tabular}
	\caption{\label{Table: V_convex vs. V_concave, T=100} Percentage difference  $100\times (V_{Concave} - V_{Convex}$) / $V_{Convex}$: Comparison of $\widetilde{V}(0,x)+A$ when demand is convex and concave (see Table \ref{Table: Intensity Functions} for definitions). The relevant parameters are $T=100, c_2 = 2\bar{c}, \gamma=0.01, c_4=c/4, \delta=0.005$. } 
\end{table}

\FloatBarrier

\subsubsection{Effect of Outside Source / Alternative Policy} The percentages in Table \ref{Table: V_{G=10^-6} - V_{G=0.01}) / V_{G=0.01}, T=100} show the effect of an outside/alternative source by presenting $\widetilde{V}(0,x)+A$ with a decreasing cost of this source $(\gamma=0.01)$ versus and a nearly constant cost ($\gamma=10^{-6}$) over time (time discount is fixed to $\delta=0.005$). The benefit of a decreasing unit cost of the outside source is observed, as the case where the cost is constant over the horizon yield higher total expected cost for every $x$ and $K$.

When $x=0, K=0$, the manufacturer may not use the alternative policy at all since the cost of procurement can be sufficiently low. When $x=0$ and $K=5000$ however, the manufacturer would prefer placing a large order at time zero, and then using the alternative policy if needed. Hence, as $K$ increases, the value of having a decreasing unit cost in the alternative policy also increases.

When $x$ is in the region (350, 450) for any $K$, it is likely that the manufacturer utilizes initial inventory and then switches to alternative policy, instead of placing an order. Hence, given the cost structure of the alternative policy, one can observe the highest percent values in the expected total cost differential in this region of $x$.

When $x=550$ or higher, the manufacturer may not use the alternative policy at all until the stopping time, since the initial inventory seems sufficiently high to cover the demand. Note that, for $x=550$, it is likely that the optimal stopping time is realized closer to $T$. Therefore, any change in the unit cost of alternative policy over time has practically no impact on the expected total cost. Of course, as $x$ goes higher (which might not be very reasonable for the problem structure), we see that the optimal solution may prefer stopping before the end of the horizon (almost at the same time for all $K$ values) and starting to use the alternative source for the remaining part of the horizon. The fact that percentages are higher simply reflect the unit cost difference in the alternative policy in the cases compared. 

\noindent \textbf{Insight 2:} It is shown in other parts of the study that the existence of an alternative policy (or an outside source) can be essential for flexibility needed in the environment. Hence, the cost of this alternative becomes critical in the effectiveness of the approach. Thus, larger percentages in Table \ref{Table: V_{G=10^-6} - V_{G=0.01}) / V_{G=0.01}, T=100} demonstrate the fact that it may be better to support the development of an alternative source so that it will become more cost-efficient (cheaper) over time. This might be realized by giving incentives to other parties for developing technologies to lower the manufacturing price. 

\begin{table}[h!]
	\centering
	\begin{tabular}{|c|ccccccc|}
		\hline
  \diagbox[width=3em]{$K$}{$x$} & 0  & 100  & 250   & 350   & 450  & 550 & 700  \\ \hline
		0                   & 1 & 1  & 1   & 2  & 3  & 0 & 5  \\
		1000                & 4 & 5  & 7   & 11 & 15 & 0 & 5  \\
		5000                & 8 & 11 & 12 & 26 & 30  & 0 & 5  \\ \hline
	\end{tabular}
	\caption{\label{Table: V_{G=10^-6} - V_{G=0.01}) / V_{G=0.01}, T=100} Percentage difference $100\times (V_{\gamma=10^{-6}} - V_{\gamma=0.01}) / V_{\gamma=0.01}$: Comparison of $\widetilde{V}(0,x)+A$ when $\gamma=10^{-6}$ and $\gamma=0.01$ to show the effect of alternative policy. The relevant parameters are $T=100, c_2 = 2\bar{c}, c_4=\bar{c}/4, \delta=0.005$, convex intensity. } 
\end{table}

\subsubsection{Effect of Time Horizon} Table \ref{Table: (V_{T=50} - V_{T=100}) / V_{T=100}} shows the effect of time horizon $T$ by presenting $\widetilde{V}(0,x)+A$ when $T=50$ and $T=100$. When $x=0, K=0$, the expected total cost under $T=100$ is lower, since the manufacturer places small orders later in time, utilizing time-discount. On the other hand, when $x=0, K=5000$, a sufficiently large order is placed at time zero. Since this purchasing cost occurs at time zero in both cases, the costs are similar.

For a relatively small range for $x$ (300-330) (for instance, when $x=331$) and large $K$, the manufacturer does not place an order and uses the alternative policy. If $T=50$, then this policy is used earlier at a time when the unit cost of the alternative policy is relatively higher and discount has less effect. This results in a 13\% difference in relative total expected costs.  On the other hand, for $x=331, K=0$, the manufacturer can place small orders instead of using the alternative policy or facing penalty. This explains a very small percentage difference observed. When $x$ is larger than the range given above, for smaller $K$ values, we start to observe the negative effects of longer horizon, as longer horizon brings more carrying cost over time and hence greater expected costs for the case with $T=100$.

\noindent\textbf{Insight 3:} The selection of time horizon which is equivalent to setting a warranty period is not considered in the current work. On the other hand, extending the warranty period will always be preferable by customers. Hence if one observes benefits of extending, it might be potentially a beneficial managerial move. If there are moderate to high setup cost $K$ values, and relatively low initial inventory $x$ values (south middle east part of Table \ref{Table: (V_{T=50} - V_{T=100}) / V_{T=100}}) it might be wiser to extend the horizon to $T = 100$ under the knowledge that demand will be flatter through 100 periods.

\begin{table}[h!]
	\centering
	\begin{tabular}{|c|cccccc|}
		\hline
	  \diagbox[width=3em]{$K$}{$x$} & 0   & 100 & 250 & 331  & 400      & 435   \\ \hline
   	     0                   & 7 & 8 & 7 & 1  & -12   & -24 \\
	  	 1000                & 3 & 5 & 6 & 3  & -6    & -14 \\
 		 5000                & 1 & 4  & 6 & 13 & 15    & -1  \\ \hline
	\end{tabular}
	\caption{\label{Table: (V_{T=50} - V_{T=100}) / V_{T=100}} Percentage difference $100\times (V_{T=50} - V_{T=100}) / V_{T=100}$: Comparison of $\widetilde{V}(0,x)+A$ when $T=50$ and $T=100$ to show the effect of time horizon. The relevant parameters are $c_2 = 2\bar{c}, c_4=\bar{c}/4,\gamma = 0.01, \delta=0.005$, convex intensity. } 
\end{table}

\FloatBarrier

\subsubsection{Effect of Penalty Cost} Table \ref{Table: V_{p=1000} - V_{p=200}) / V_{p=200}, G=0.01} and Table \ref{Table: V_{p=1000} - V_{p=200}) / V_{p=200}, G=10^(-6)} show the effect of penalty $\bar{c}_2$ by presenting $\widetilde{V}(0,x)+A$ when $\bar{c}_2=2\bar{c}$ and $\bar{c}_2=10\bar{c}$. Note that for $x$ values which are in between 0 and the expected total demand (for instance $x=300$), the change in unit penalty cost is expected to have its highest impact, since the firm takes the risk of penalty for not placing an order. Nevertheless, even if we change penalty cost by a factor of 5, the increase in the optimal value of the expected total discounted cost is negligible (less than 2\% in all cases). The reason is that the manufacturer can stop holding inventory and use the alternative policy to avoid penalty cost.

\noindent\textbf{Insight 4:} As we have an existing alternative (which is much cheaper than the larger penalty cost in Table \ref{Table: Our numerical Values}), practically there is no significant difference observed after changing the penalty cost. Hence, with the existence of such an alternative, the firm might announce to pay large penalties for not satisfying demand to attract more demand to begin with. This shows the importance of creating such an alternative. Additionally, if the cost of alternative decreases over time (periods where the risk of paying the penalty is more), then it will be even better for decreasing expected costs.

\begin{table}[h!]
	\centering
	\begin{tabular}{|c|cccc|}
		\hline
	 \diagbox[width=3em]{$K$}{$x$}	& 0     & 100   & 250   & 300   \\ \hline
		0    & 0.4 & 0.5 & 0.7 & 0.8 \\
		1000 & 0.5 & 0.7 & 0.9 & 1.1 \\
		5000 & 0.6 & 0.8 & 1.0 & 1.8 \\ \hline
	\end{tabular}
	\caption{\label{Table: V_{p=1000} - V_{p=200}) / V_{p=200}, G=0.01} Percentage difference $100\times (V_{\bar{c}_2=10\bar{c}} - V_{\bar{c}_2=2\bar{c}}) / V_{\bar{c}_2=2\bar{c}}$: Comparison of $\widetilde{V}(0,x)+A$ when $\bar{c}_2=2\bar{c}$ and $\bar{c}_2=10\bar{c}$. The relevant parameters are $T= 100, c_4=\bar{c}/4,\gamma = 0.01, \delta=0.005$, convex intensity. } 
\end{table}

\begin{table}[h!]
	\centering
	\begin{tabular}{|c|cccc|}
		\hline
		& 0     & 100   & 250   & 600   \\ \hline
		0    & 0.4 & 0.5 & 0.6 & 0.0 \\
		1000 & 0.4 & 0.5 & 0.7 & 0.0 \\
		5000 & 0.3 & 0.6 & 0.7 & 0.0 \\ \hline
	\end{tabular}
	\caption{\label{Table: V_{p=1000} - V_{p=200}) / V_{p=200}, G=10^(-6)} Percentage difference $100\times (V_{\bar{c}_2=10\bar{c}} - V_{\bar{c}_2=2\bar{c}}) / V_{\bar{c}_2=2\bar{c}}$: Comparison of $\widetilde{V}(0,x)+A$ when $\bar{c}_2=2\bar{c}$ and $\bar{c}_2=10\bar{c}$. The relevant parameters are $T= 100, c_4=\bar{c}/4,\gamma = 10^{-6}, \delta=0.005$, convex intensity. } 
\end{table}

\subsubsection{Effect of Time Discount} Table \ref{Table: V_{delta=10^{-6}} - V_{delta=0.005}) / V_{delta=0.005}} shows the effect of time discount $\delta$ by presenting $\widetilde{V}(0,x)+A$ when $\delta=0.005$ and $\delta=10^{-6}$. The expected total cost is always higher when time discount is close to zero, as can be predicted. The effect of discount decreases in $K$, for smaller $x$ values, since a large order is placed at time zero rather than later on. However, for intermediate $x$ values (350-450) the effect is reversed or disappears, since it is likely that an order is needed later in the horizon; hence, the value of $K$ becomes critical.

\noindent\textbf{Insight 5:} Time discount shows the sensitivity of results on the total discounted expected cost for varying horizon lengths. In the numerical experiments, the time discount value seems to be effective for different $x$ values rather than $K$; and hence reiterating the importance of initial inventory. This length will be different for industries, and hence the essential insight will be a function of the industry considered. 

\begin{table}[h!]
	\centering
	\begin{tabular}{|c|ccccc|}
		\hline
		\diagbox[width=3em]{$K$}{$x$} & 0    & 100  & 250  & 350  & 450  \\ \hline
		0    & 11 & 13 & 18 & 20 & 17 \\
		1000 & 8  & 11 & 15 & 19 & 18 \\
		5000 & 6  & 8  & 14 & 21 & 18 \\ \hline
	\end{tabular}
	 \caption{ \label{Table: V_{delta=10^{-6}} - V_{delta=0.005}) / V_{delta=0.005}} Percentage difference $100\times (V_{\delta=10^{-6}} - V_{\delta=0.005}) / V_{\delta=0.005}$: Comparison of $\widetilde{V}(0,x)+A$ when $\delta=0.005$ and $\delta=10^{-6}$. The relevant parameters are $T= 100, c_2 = 2\bar{c}, c_4=\bar{c}/4,\gamma = 0.01$, convex intensity.  } 
\end{table}

\subsubsection{Effect of Scrapping Cost}

Table \ref{Table: V_{scr=c/4} - V_{scr=-c/4}) / V_{scr=-c/4}} shows the impact of scrapping cost $c_4$ by presenting $\widetilde{V}(0,x)+A$ when $c_4=\bar{c}/4$ and $c_4=-\bar{c}/4$. If $x<450$, then it is most likely that the inventory is used to satisfy the demand; therefore, we may not need to scrap inventory. In this region of $x$, we stop holding inventory only if the inventory level is about to hit zero and the risk of penalty arises. In such a case, only a negligible amount of inventory is scrapped, implying that the scrapping cost has a negligible effect on the expected total cost. On the other hand, if $x>450$, then the excess inventory may need to be scrapped. Therefore, the scrapping cost can have an impact on the expected total cost.

\begin{table}[h!]
	\centering
	\begin{tabular}{|c|cccccc|}
		\hline
		\diagbox[width=3em]{$K$}{$x$} & 0     & 100   & 250   & 450   & 500   & 550    \\ \hline
		0    & 0.0 & 0.0 & 0.0 & 0.1 & 2.8 & 15.1 \\
		1000 & 0.1 & 0.1 & 0.1 & 0.2 & 2.8 & 15.1 \\
		5000 & 0.1 & 0.1 & 0.1 & 0.2 & 2.8 & 15.1 \\ \hline
	\end{tabular}
	\caption{\label{Table: V_{scr=c/4} - V_{scr=-c/4}) / V_{scr=-c/4}} Percentage difference $100\times (V_{c_4=\bar{c}/4} - V_{c_4=-\bar{c}/4}) / V_{c_4=-\bar{c}/4}$: Comparison of $\widetilde{V}(0,x)+A$ when $c_4=\bar{c}/4$ and $c_4=-\bar{c}/4$. The relevant parameters are $T= 100, c_2 = 2\bar{c},\gamma = 0.01$, $\delta = 0.005$, convex intensity.} 
\end{table}

\noindent\textbf{Insight 6:} Scrapping inventory is not a significant burden (or source of income) when the manufacturer wishes to start using outside source rather than holding inventory, unless the initial inventory level is excessively large. So, this parameter seems to be less effective for decision-making purposes.

\subsubsection{Expected Penalty of a Misspecified Intensity Function}

We analyze the impact of an error in selecting the intensity function of the non-homogeneous Poisson process $(N_t)_{t\in[0,T]}$. Suppose that the manufacturer chooses a linear intensity function, but the true intensity function is convex (recall Table \ref{Table: Intensity Functions} for definitions). To calculate the cost of making such assumption, we first solve $\widetilde{V}(0,x)+A$ when the intensity function is convex and linear to find the best decision variables (ordering, stopping, continuation regions and order-up-to levels). Next, with those fixed decision variables, we compute the objective function $\widetilde{V}(0,x)+A$ when intensity is convex.

Table \ref{Table: V_linear_decvar vs. V_convex_decvar, T=50} shows the percent difference in expected total cost. If $x=0,K=0$ and intensity is linearly decreasing, then the manufacturer places small orders more frequently. On the other hand, if the intensity is decreasing in a convex manner, then the manufacturer tends to place larger orders at the beginning and smaller orders towards the end. Hence, by presuming a linearly decreasing intensity function and taking actions based on this assumption, the manufacturer can observe excess penalty cost at the beginning and excess holding cost towards the end, resulting in a significant loss (above 30\%). We believe that this is a motivation to study a problem where intensity rate is random itself. Also note that for other $x$ and $K$ combinations, the loss can be as high as 116\%. 

The largest percent loss can be observed when there is an initial inventory level which is close to the expected demand throughout the horizon. For instance when $x=450$, it is likely that we wait for an amount of time and then place an order. If intensity is linearly decreasing, then this future order can be large. Hence, if the manufacturer places a large order in the future, then it is most likely that excess holding and procurement costs incur as the arrival rate under decreasing-convex case is much lower towards the end of horizon.

When $x$ is large (above expected demand), one may be less willing to stop towards the end under the presumption that the intensity rate is decreasing linearly. Hence, this difference in the stopping region create small, but meaningful percent difference in expected total costs indicating the importance of the stopping time. 

\noindent\hypertarget{Penalty of Misspecification}{\textbf{Insight 7:}} One of the critical inputs to the end-of-life management problem is the  estimation of the demand rate over the time horizon. The results presented in Table \ref{Table: V_linear_decvar vs. V_convex_decvar, T=50} indicate that the penalty of this misspecification can be drastic, especially if the fixed ordering cost is large. Hence an initial study to analyze underlying demand structure seems to be a reasonable way for management to use her resources. Similarly, an agreement with the consumers on the possible timing of demand arrivals may further help in quantifying the demand intensity over time,  decreasing the risk of misspecification. 

\begin{table}[h!]
	\captionsetup{font=small}
	\centering
	\begin{tabular}{|c|cccccc|}
		\hline
		& 0    & 100  & 250  & 450   & 550 & 650 \\ \hline
		0    & 36 & 22 & 6  & 10  & 0 & 3 \\
		1000 & 11 & 14 & 16 & 85  & 1 & 3 \\
		5000 & 29 & 5  & 36 & 116 & 0 & 3 \\ \hline
	\end{tabular}
	\caption{ \label{Table: V_linear_decvar vs. V_convex_decvar, T=50}  Percentage difference $100\times(V_{linear-decvar} - V_{convex-decvar}$) / $V_{convex-decvar}$: The cost $V_{linear-decvar}$ denotes the expected total cost of $\widetilde{V}(0,x)+A$ calculated with convex intensity, but the best decision variables are found with linear intensity. $V_{Convex-decvar}$ denotes the expected total cost of $\widetilde{V}(0,x)+A$ with convex intensity. The relevant parameters are $T=50, c_2 = 2\bar{c}, \gamma=0.01, c_4=c/4, \delta=0.005$.  } 
\end{table}
\section{Conclusions with Practical Implications and Possible Extensions} \label{Section: Conclusion}

This study analyzes the value of providing flexibility in the end-of-life management problem. Namely, we allow for multiple orders as well as a change in the timing of the first order, and we utilize stopping times to decide on when to dispose all the available inventory. To that end, we consider a manufacturer whose problem is to make one of the three decisions at each period: (1) place an order for spare parts, (2) do nothing and use existing inventory to satisfy demand, or (3) stop holding inventory permanently and use an outside/alternative source. We cast this problem as an optimal stopping problem with additional decisions so that it can be solved by means of stochastic dynamic programming. After providing the dynamic programming formulation, we use martingale theory to facilitate the calculation of the value function. We devise a taxonomy for benchmark models to show the value of our approach as well as compare our results with the current literature. Several analytical results are presented to further enhance our understanding of the problem. Finally, we present computational results, generating several managerial insights. 

The originality of the study comes from the fact that several dimensions of the end-of-life problem are considered. The first dimension is related to the decision-making environment. We consider possible decisions which give us the benefit of using all flexibilities. Remarks \mbox{\ref{Remark: Allowing multiple orders is important}} through \mbox{\ref{Remark: Considering joint effects is important}} summarize when and how to exploit these flexibilities. Accordingly, approaches with the premise that a final order must be placed at time zero can be a strong assumption leading to losses. Moreover, the dynamic selection of the time to stop (via a stopping time) and additionally allowing for multiple orders can be valuable.

A second dimension is related to the practicalities during the implementation phase of these policies. Note that the problem of managing end-of-life is more than supplying spare parts inventory, especially if the horizon is long. Hence, operationally, additional information we obtain during the horizon can be used more effectively to manage the decisions within the horizon. More precisely, one can compute several properties regarding the stopping time, as depicted in Subsection \mbox{\ref{Subsection: Other Useful Structural Results}}. The information gathered would allow proactive agreements with the outside source, as well as some operational support in the later phases of the end-of-life. The manufacturer can share the information on the time of  transition with the outside source. The information can be based on some assumptions (such as no further order will be given, as in Subsections \mbox{\ref{Subsubsection: Bound on the Switching Time}} and \mbox{\ref{Subsubsection: S is increasing in tau}}) or can be in the form of exact distribution (Subsection \mbox{\ref{Subsubsection: Distribution of Stopping Time}}). Although not presented here, using the analysis carried out in Appendix \ref{Sec: Computation of Distribution of stopping time} (distribution of stopping times), one can come up with various additional information to further support decision-making in all, but especially the later stages of the horizon. Two straightforward examples are as follows: (1) One can compute the probability that an order is not placed in the next $n$ periods. For instance, the manufacturer can use this probability for the following decision: if the probability is greater than a threshold, then it might be beneficial not to review inventory for some number of periods, avoiding review costs. Note that the review cost component affects the selection of the time between two consecutive periods. (2) For a fixed initial inventory level, the expected number of periods before placing the first order can be calculated. This information will likely be utilized by the manufacturing function for planning the production of this lot in advance. 

A final dimension we consider is regarding the setting of the parameters, which is likely to affect the problem outcome significantly. In Subsection \ref{Subsection: Sensitivitiy Analysis}, seven managerial insights are proposed. The insights are mainly towards controlling the environment of the end-of-life management problem. They are related to affecting the customer arrival rates, making monetary arrangements to support development of the alternative source, extending the warranty period, and announcing favorable parameter values to attract more customers so that demand rates increase, making the expected profits of the end-of-life period even more attractive for the manufacturing firm. 

There are a few straightforward extensions that can follow: use of time-varying unit procurement cost, use of time-varying fixed ordering cost, use of unequal review periods, and use of costs to review inventory. Except the last one, the current dynamic programming formulation can be adjusted. For the case with review costs, state space-reducing properties can be studied. Finally, observing Insight \hyperlink{Penalty of Misspecification}{7}, which is related to the significant cost of misspecification of the demand intensity function, the case with random intensity might be a reasonable future direction for research.

\appendix
\section{Auxiliary Results}

This subsection provides auxiliary results for the other proofs. Recall that $(\Omega,\cH,\PP)$ denotes the underlying probability space which hosts the non-homogeneous Poisson process $N$. Let $\FF=(\cF_t)_{t\in[0,T]}$ denote the filtration generated by $N$, that is, $\cF_t:=\sigma(N_s, s\in[0,t])$ for every $t\in[0,T]$. Lemma \ref{Lemma: Martingale of NHPP} below introduces the martingale property for $N$ and it helps us convert Poisson integrals into Lebesgue integrals. 
\begin{lem} \label{Lemma: Martingale of NHPP}
	\citep[p.~299, VI.6.4]{ccinlar2011probability} Let $(H_t)_{t\in[0,T]}$ be a non-negative $\FF$-predictable process such that $\E\sqb{\int_{0}^{t} H_u\lambda(u)du} < \infty$ for every $t\in[0,T]$. 
	Then, the process $(L_t)_{t\in[0,T]}$ defined by
	\begin{equation*}
	L_t := \int_{0}^{t} H_u\,dN_u - \int_{0}^{t} H_u\lambda(u)\,du, \quad t\in[0,T],
	\end{equation*}
	is a martingale with respect to $\FF$. Moreover, for each $\FF$-stopping time $\tau\in\cT$,
	\begin{equation*}
	\E\sqb{\int_{0}^{\tau} H_u\,dN_u} = \E\sqb{\int_{0}^{\tau}H_u\lambda(u)\,du}.
	\end{equation*}
\end{lem}

The next lemma is helpful while converting the expected cost terms into new forms that can be calculated numerically.
\begin{lem} \label{Lemma: E[x-N(tau)^+] = sum P(N(tau) <= k)}
	For every $x \in \ZZ_+$ and $k\in\TT$,
	\[
	\mathbb{E} [ (x-N_k)^+] = \sum_{n=0}^{x-1} \PP\cb{N_k \leq n},
	\]
	where the sum is defined to be 0 when $x=0$. 
\end{lem}
\begin{proof}
	Note that $\mathbb{E}[(x-N_k)^+] - \mathbb{E}[(x-1-N_k)^+]=\PP\cb{N_k \leq x-1}$. Iterating this equality yields that $\mathbb{E}[(x-N_k)^+] = \mathbb{E}[(x-1-N_k)^+] + \PP\cb{N_k \leq x-1} = \mathbb{E}[(x-2-N_k)^+] + \PP\cb{N_k \leq x-2} + \PP\cb{N_k \leq x-1} = \dots = \sum_{n=0}^{x-1} \PP\cb{N_k \leq n}.$
\end{proof}
Recall that $\sigma^{k}_x=\inf\{u>k: N_u-N_k\geq x\}$ and $\sigma_x = \sigma^0_x$. The following lemma shows that the stopping time $\sigma_x$ is conditionally independent from the past given the present. This enables us to write dynamic programming algorithms when the objective function includes $\sigma_x$. For an $\FF$-stopping time $\tau$, we define the stopped $\sigma$-algebra $\cF_\tau:= \cb{E\in\cH : E\cap \cb{\tau\leq t}\in\cF_t \text{ for each } t\in[0,T]}$.
 
\begin{lem} \label{Lemma: sigma_x has two parts}
	For every $\FF$-stopping time $\tau\in\cT$ and $x \in \ZZ_+$, on $\{\tau < \sigma_x\}$,
	\begin{equation*}
	\sigma_x = \sigma_{x-N_\tau}^{\tau} + \tau,
	\end{equation*}
	where $\sigma_{x-N_\tau}^{\tau} = \inf \{t>0: N_{t+\tau}-N_\tau \geq x-N_\tau \}$. Moreover, $\sigma_{x-N_\tau}^{\tau}$ and $\cF_\tau$ are conditionally independent given $N_\tau$.
\end{lem}

\begin{proof}
	On the set $\{\tau < \sigma_x\}$, we see that $N_\tau < x$. Then, on $\{\tau < \sigma_x \}$, we have
	\begin{align*}
	\sigma_x 
	= & \inf \{t>0:N_t \geq x \} 
	= \inf \{t > \tau: N_t \geq x \} \\
	= & \inf \{u>0: N_{u+\tau}-N_\tau+N_\tau \geq x \} + \tau 
	= \sigma_{x-N_\tau}^{\tau} + \tau.
	\end{align*}
	For the second claim, we know from the strong Markov property of non-homogeneous Poisson processes \citep[p.\@ 296, VI.5.18]{ccinlar2011probability} that $N_{u+\tau}$ and $\cF_\tau$ are conditionally independent given $N_\tau$ for any $u>0$. Hence, by definition, $\sigma^\tau_{x-N_\tau}$ and $\cF_\tau$ are conditionally independent given $N_\tau$.
\end{proof}

\section{Proofs of the Results in Section \ref{Section: Problem Definition}} \label{Section: Proofs of Section Problem Defn}

\begin{proof}[Proof of \autoref{Lemma: Holding Cost}] \label{Proof: Lemma: Holding cost}
	Clearly, $H(k,0)=0$. Let $x\geq 1$. Then,
	\begin{align*}
	H(k,x) = & c_1\E\sqb{\int_{k}^{k+1} e^{-\delta(u-k)} \big( x-(N_u-N_k)\big)^+\,du } \\
	= & c_1\int_{k}^{k+1} e^{-\delta(u-k)} \E\sqb{\big( x-(N_u-N_k)\big)^+}\,du \quad\text{(Fubini's Theorem)}\\
	= & c_1\int_{k}^{k+1} e^{-\delta(u-k)} \sum_{n=0}^{x-1}\PP\cb{N_u-N_k\leq n} \,du \quad\text{(Lemma \ref{Lemma: E[x-N(tau)^+] = sum P(N(tau) <= k)})} \\
	= & c_1\sum_{n=0}^{x-1} \int_{k}^{k+1} e^{-\delta(u-k)} \PP\cb{N_u-N_k\leq n} \,du \\
	= & c_1\sum_{n=0}^{x-1} \sum_{i=0}^{n} \int_{k}^{k+1} e^{-\delta(u-k)} \PP\cb{N_u-N_k=i} \,du \\
	= & c_1\sum_{n=0}^{x-1} \sum_{i=0}^{n} \int_{k}^{k+1} e^{-\delta(u-k)} e^{-(\Lambda(u)-\Lambda(k))} \frac{(\Lambda(u)-\Lambda(k))^i}{i!} \,du,
	\end{align*}
	where $\Lambda(u)-\Lambda(k)=\int_{k}^{u}\lambda(s)ds$.
\end{proof}

\begin{proof}[Proof of \autoref{Lemma: Lost sales cost}]
We first note that $L(k,x)$ can be expressed as
\begin{equation*}
L(k,x) = \EE\sqb{\int_{k}^{k+1} e^{-\delta(u-k)} c_2(u)dN_u} - \EE\sqb{\int_k^{(k+1)\wedge\sigma_x} e^{-\delta(u-k)} c_2(u)dN_u }.
\end{equation*}
It follows from Lemma \ref{Lemma: Martingale of NHPP} that
\[
\E\sqb{\int_{k}^{k+1} e^{-\delta(u-k)} c_2(u) \,dN_u} 
=  \int_{k}^{k+1} e^{-\delta (u-k)} c_2(u) \lambda(u)\,du.
\]
Moreover, we have
\begin{align*}
& \E\sqb{\int_{k}^{(k+1) \wedge \sigma^k_x} e^{-\delta(u-k)} c_2(u) \,dN_u} \\
= & \E\sqb{\int_{k}^{(k+1) \wedge \sigma^k_x} e^{-\delta(u-k)} c_2(u)\lambda(u)\,du} \quad\text{(Lemma \ref{Lemma: Martingale of NHPP})} \\
= & \int_{k}^{k+1} \E\sqb{ 1_{\cb{u<\sigma^k_x}} } e^{-\delta(u-k)} c_2(u)\lambda(u)\,du  \quad\text{(Fubini's Theorem)}\\
= & \int_{k}^{k+1} \PP\cb{N_u-N_k\leq x} e^{-\delta(u-k)} c_2(u)\lambda(u)\,du \quad\text{ (Definition of $\sigma^k_x$)} \\
= & \sum_{i=0}^{x} \int_{k}^{k+1} e^{-\delta(u-k)} c_2(u)\lambda(u) e^{-(\Lambda(u)-\Lambda(k))}\frac{(\Lambda(u)-\Lambda(k))^i}{i!} \,du.
\end{align*}
Hence, the result of the lemma follows. 
\end{proof}

\section{Proof of the Results in Subsection \ref{Subsection: Reformulation}} \label{Section: Proofs of Subsection Reformulation}

\begin{proof}[Proof of \autoref{Prop: Reformulate Obj. Fnc.}]
	Note that for every $t\in\TT$ and every $x_t\in\ZZ_+$, we can write
	\begin{align*}
	e^{-\delta t} S(t,x_t) = & e^{-\delta t} c_{4}x_t + e^{-\delta t} \E\sqb{\int_{t}^{T} e^{-\delta(u-t)} c_{3}(u) dN_u } \\
	= & e^{-\delta t} c_{4}x_t + \E\sqb{\int_{t}^{T} e^{-\delta u} c_{3}(u) dN_u } \\
	= & e^{-\delta t} c_{4}x_t + \E\sqb{\int_{0}^{T} e^{-\delta u} c_{3}(u) dN_u } - \E\sqb{\int_{0}^{t} e^{-\delta u} c_{3}(u) dN_u } \\
	= & e^{-\delta t} c_{4}x_t + \E\sqb{\int_{0}^{T} e^{-\delta u} c_{3}(u) dN_u } - \sum_{k=0}^{t-1} e^{-\delta k} \EE\sqb{\int_k^{k+1} e^{-\delta (u-k)} c_{3}(u) dN_u} \\
	= & e^{-\delta t} c_{4}x_t + A - \sum_{k=0}^{t-1} e^{-\delta k} \EE\sqb{\int_k^{k+1} e^{-\delta (u-k)} c_{3}(u) dN_u}.
	\end{align*}
	Moreover, for every $t\in\TT$ and every $x_0,x_1,\dots,x_{t-1}\in\ZZ_+$, we can write
	\begin{align*}
	\sum_{k=0}^{t-1} e^{-\delta k} L(k,x_k) = & \sum_{k=0}^{t-1} e^{-\delta k} \EE \sqb{\int_{ (k+1)\wedge \sigma^{k}_{x_k}}^{k+1} e^{-\delta (u-k)} c_2(u) dN_u}\\
	= & \sum_{k=0}^{t-1} e^{-\delta k} \of{ \EE \sqb{\int_{k}^{k+1} e^{-\delta (u-k)} c_2(u) dN_u} - \EE \sqb{\int_{k}^{(k+1)\wedge \sigma^{k}_{x_k} } e^{-\delta (u-k)} c_2(u) dN_u}}.
	\end{align*}
	Then, for every $t\in\TT$ and every $x_0,x_1,\dots,x_t\in\ZZ_+$, we have
	\begin{align*}
	& \sum_{k=0}^{t-1} e^{-\delta k} C(k,x_k) + e^{-\delta t} S(t,x_t) \\
	= & \sum_{k=0}^{t-1} e^{-\delta k} \of{H(k,x_k) + L(k,x_k)} + e^{-\delta t} S(t,x_t) \\
	= & \sum_{k=0}^{t-1} e^{-\delta k} c_1\EE\sqb{\int_k^{k+1} e^{-\delta (u-k)} (x_k- (N_u-N_k))^+ du} \\
	& + \sum_{k=0}^{t-1} e^{-\delta k} \of{ \EE \sqb{\int_{k}^{k+1} e^{-\delta (u-k)} c_2(u) dN_u} - \EE \sqb{\int_{k}^{(k+1)\wedge \sigma^{k}_{x_k} } e^{-\delta (u-k)} c_2(u) dN_u}} \\
	& + e^{-\delta t} c_{4}x_t + A - \sum_{k=0}^{t-1} e^{-\delta k} \EE\sqb{\int_k^{k+1} e^{-\delta (u-k)} c_{3}(u) dN_u} \\
	= & \sum_{k=0}^{t-1} e^{-\delta k} c_1\EE\sqb{\int_k^{k+1} e^{-\delta (u-k)} (x_k- (N_u-N_k))^+ du} \\
	& + \sum_{k=0}^{t-1} e^{-\delta k} \of{ \EE \sqb{\int_{k}^{k+1} e^{-\delta (u-k)} [c_2(u)-c_3(u)] dN_u} - \EE \sqb{\int_{k}^{(k+1)\wedge \sigma^{k}_{x_k} } e^{-\delta (u-k)} c_2(u) dN_u}}\\
	& + e^{-\delta t} c_{4}x_t + A\\
	= & \sum_{k=0}^{t-1} e^{-\delta k} \widetilde{C}(k,x_k) + e^{-\delta t} \widetilde{S}(x_t) + A.
	\end{align*}
	Therefore, for every $\pi\in\Pi$, we have
	\begin{align*}
	& \left. \E\sqb{\sum_{k=0}^{\tau-1} e^{-\delta k} \bigg( c(\mu_k(X_k))+ C(k,X_k+\mu_k(X_k))\bigg) + e^{-\delta\tau} S(\tau,X_\tau) \right\vert X_0 = x}  \\
	= & \left. \E\sqb{\sum_{k=0}^{\tau-1} e^{-\delta k} \bigg( c(\mu_k(X_k))+ \widetilde{C}(k,X_k+\mu_k(X_k))\bigg) + e^{-\delta\tau} \widetilde{S}(X_\tau) \right\vert X_0 = x}  + A.
	\end{align*} 
	After taking infimum over all $\pi\in\Pi$, we conclude the proof.
\end{proof}

\begin{proof}[Proof of \autoref{Proposition: Converting tilde(C)}]
	Let $k\in\TT$ and $x\in\ZZ_+$. If $x=0$, then $\sigma^k_x = k$ and $\of{x-(N_u-N_k)}^+ = 0$ for every $u>k$. Therefore, we apply Lemma \ref{Lemma: Martingale of NHPP} to the definition of $\widetilde{C}(k,x)$ in \eqref{Equation: tilde(C)(k,x)} to get
	\begin{align*}
	\widetilde{C}(k,0) & = \EE\sqb{\int_k^{k+1} e^{-\delta(u-k)} \sqb{c_2(u)-c_3(u)} dN_u}  = \int_k^{k+1} e^{-\delta(u-k)} \sqb{c_2(u)-c_3(u)} \lambda(u)\,du.
	\end{align*}
	Next, if $x\geq 1$, then from Lemma \ref{Lemma: Holding Cost}, we have
	\[
	\EE\sqb{\int_k^{k+1} e^{-\delta(u-k)} \of{ x-(N_u-N_k)}^+ du} = \sum_{n=0}^{x-1} \sum_{i=0}^{n} \int_{k}^{k+1} e^{-\delta(u-k)} e^{-(\Lambda(u)-\Lambda(k))} \frac{(\Lambda(u)-\Lambda(k))^i}{i!} \,du.
	\]
	Moreover, from Lemma \ref{Lemma: Martingale of NHPP}, we have 
	\[
	\EE\sqb{ \int_{k}^{k+1} e^{-\delta(u-k)} \sqb{c_2(u) -c_3(u)} \, dN_u} = \int_{k}^{k+1} e^{-\delta (u-k)} [c_2(u)-c_{3}(u)] \lambda(u)\,du.
	\]
	Finally, by proceeding as in the proof of Lemma \ref{Lemma: Lost sales cost}, it is possible to show that
	\[ 
	\EE\sqb{\int_k^{(k+1)\wedge \sigma_x^k} e^{-\delta(u-k)} c_2(u) \, dN_u} = \sum_{i=0}^{x} \int_{k}^{k+1} e^{-\delta(u-k)} c_2(u) \lambda(u) e^{-\Lambda(u)-\Lambda(k)}\frac{(\Lambda(u)-\Lambda(k))^i}{i!} \,du.
	\]
	Combining the above terms concludes the proof.
\end{proof}

\section{Proof of the Results in Subsection \ref{Subsection: Structural Results for D8F}} \label{Section: Proofs of Subsection Structural Results for D8F}

\begin{proof}[Proof of \autoref{Proposition: Solution to D8F}] \textcolor{white}{asd}\\
	\textbf{1.} To show the partitioning of $\TT\times\ZZ_+$ into three regions, let an arbitrary time $k\in\TT$ and an inventory level $x\in\ZZ_+$ be given (recall that $\TT=\cb{0,1,\dots,T}$). \\
	\underline{Case 1:} Assume that
	$
	\widetilde{S}(x)\leq\tJ(k,x)=\inf_{m\in\ZZ_+}(c(m)+\tG(k,x+m)).
	$ 
	Then $x\in R^S_k$ by definition. It also holds that $\tS(x)\leq \tG(k,x)$ because $c(0)=0$. Then, $x\notin R^C_k$. Moreover, it holds that $\tS(x)\leq \inf_{m\geq 1}(c(k)+\tG(k,x+m))$, so $x\notin R^O_k$. \\
	\underline{Case 2:} Assume that $\widetilde{S}(x) > \tJ(k,x) = \inf_{m\in\ZZ_+}(c(m)+\tG(k,x+m))$ and that $\widetilde{G}(k,x) \leq \inf_{m\geq 1}(c(m)+\tG(k,x+m))$. Then, from \eqref{Equation: tilde(V)(k,x_k) DP Definition}, it is possible to see that $\tV(k,x)=\tG(k,x)$. Moreover, $\tG(k,x) < \tS(x)$ due to the assumption of this case. Hence, $x\in R^C_k$.  Next, $x\notin R^S_k$ since $\tS(x)> \tG(k,x)$.  Finally, $\tG(k,x)\leq \inf_{m\geq 1}(c(m)+\tG(k,x+m))$ and therefore $x\notin R^O_k$.  \\
	\underline{Case 3:} Assume that $\widetilde{S}(x) > \tJ(k,x) = \inf_{m\in\ZZ_+}(c(m)+\tG(k,x+m))$ and that $\tG(k,x) > \inf_{m\geq 1}(c(m)+\tG(k,x+m))$. Then, $\inf_{m\geq 1}(c(m)+\tG(k,x+m)) < \min \{\tS(x),\tG(k,x)\}$ and so $x\in R^O_k$. Moreover, from \eqref{Equation: tilde(V)(k,x_k) DP Definition}, it is possible to see that $\tV(k,x)<\tG(k,x)$ so $x\notin R^C_k$. Finally, we have $x\notin R^S_k$ since $\widetilde{S}(x) > \inf_{m\in\ZZ_+}(c(m)+\tG(k,x+m))$ from the assumption of the case.\\
	\noindent\textbf{2.} Let $\tpi^*$ be given as in the proposition. For ease of notation, define for each $k\in\TT$ and $x\in\ZZ_+$ the function $\tJ_1(k,x) := \inf_{m\geq 1}(c(m)+\tG(k,x+m))$. If $X^{\tpi^*}_k\in R^S_k$, then $\tS(X_k^{\tpi^*})\leq \tJ(k,X_k^{\tpi^*})$. Then, we stop due to the definition of $\widetilde{\tau}^*$ in Corollary \ref{Cor: DP solves tilde(V)*}. If $X_k^{\tpi^*}\in R^O_k$, then $\tJ_1(k,X_k^{\tpi^*}) < \min\{ \tS(X_k^{\tpi^*}), \tG(k,X_k^{\tpi^*})\}.$ Then, from \eqref{Equation: tilde(V)(k,x_k) DP Definition}, we have $\tV(k,X_k^{\tpi^*})=\tJ_1(k,X_k^{\tpi^*}),$ so it is optimal to place an order of $\widetilde{\mu}^*_k\circ X_k^{\tpi^*} \geq 1$ where the function $\widetilde{\mu}^*_k$ is defined as in Corollary \ref{Cor: DP solves tilde(V)*}. Hence, it is optimal to increase inventory to the level $S^*_k\circ X_k^{\tpi^*}$ where $S^*_k(x):=x+\widetilde{\mu}^*_k(x)$ for each $x\in R^O_k$. Finally, if $X_k^{\tpi^*}\in R^C_k$, then $\tV(k,X_k^{\tpi^*}) =\tG(k, X_k^{\tpi^*})$, so the decision of not placing an order ($\widetilde{\mu}^*_k\circ X_k^{\tpi}=0$ ) can attain the infimum in the definition of $\tJ$ in \eqref{Equation: tilde(J)(k,x)}. Moreover, $X_k^{\tpi^*}\notin R^O_k$ and $X_k^{\tpi^*}\notin R^S_k$ since $R^C_k, R^O_k$ and $R^S_k$ are disjoint from the first claim of the proposition. Then, it is optimal to continue without taking an action. 
\end{proof}

\section{Distribution of Optimal Stopping Time}\label{Sec: Computation of Distribution of stopping time}
In this section, we provide recursive relations to compute the distribution of optimal stopping time. Recall that $\pi^*$, $\mu^*_k$ $\tau^*$ and $X^{\pi^*}_k$ respectively denote an optimal policy, optimal order amount, optimal stopping time, and inventory level. Let us ease the notation and use $\pi,\mu_k,\tau$ and $X_k$ in this section. Let $x_0$ denote the initial inventory level at time $0$. Then, we stop at time $m$ if the first time that we enter the stopping region is $m$. Formally,
\begin{align}\label{Eq: tau=t}
\PP\cb{\tau= m} = \PP\cb{\condition{ X_1 \notin R^S_1, \dots,\, X_{m-1}\notin R^S_{m-1},\, X_m \in R^S_m } X_0=x_0 }.
\end{align}
where $R^S_k$ denotes the set of inventory levels that we stop at time $k$. In this section, we provide a recursive relation to compute \eqref{Eq: tau=t}. Define
\begin{equation}\label{Eq: P(t,x_t)}
P(t,x_t) := \PP\cb{\condition{  X_{t+1} \notin R^S_{t+1},\dots,\, X_{m-1}\notin R^S_{m-1},\, X_m \in R^S_m } X_t = x_t }.
\end{equation}
Then, $\PP\cb{\tau=m} = P(0,x_0)$ and we calculate $P(0,x_0)$. To that end, let us first state the following lemma.
\begin{lem} \label{Lem: P(a,b,c|d) = P(a,b|d)+P(a,c|d)}
	Let $A,B,C,D$ be non-negligible events such that $B\cap C=\emptyset$. Then,
	\begin{equation*}
	\PP(A\cap (B\cup C) | D) = \PP(A\cap B|D) + \PP(A\cap C|D).  
	\end{equation*}
\end{lem}
\begin{proof}
	$B\cap C=\emptyset$ implies $(A\cap B\cap D)\cap (A\cap C \cap D)=\emptyset$. Then,
	\begin{align*}
	\PP(A\cap (B\cup C) | D ) 
	= & \frac{1}{\PP(D)} \PP(A\cap (B\cup C)\cap D) \\
	= & \frac{1}{\PP(D)} \sqb{ \PP(A\cap B\cap D) + \PP(A\cap C\cap D)  } \\
	= & \PP(A\cap B|D) + \PP(A\cap C|D). 
	\end{align*}
\end{proof}

The following proposition enables us to calculate $P(0,x_0)$. 
\begin{prop}
	The function $P$ can be expressed recursively by
	\begin{equation*}
	P(t-1,x) = \sum_{\substack{n\in\ZZ_+: \\ x-n\in R^C_{t}}} P(t, x-n) \PP\cb{N_t - N_{t-1} =n} + \sum_{\substack{ n\in\ZZ_+ :\\ x-n \in R^O_t}} P(t, y_t) \PP\cb{N_t - N_{t-1}=n}
	\end{equation*}
	with the terminal condition $P(m,x_m)=1$ if $x_m\in R^S_m$ and $P(m,x_m)=0$ if $x_m\notin R^S_m$. Here, $y_t$ denotes the order-up-to level at time $t$. 
\end{prop}
\begin{proof}
	Let $N_{t}^{t-1}:= N_t - N_{t-1}$ be the total demand during one period. The following equation relates the inventory levels between $t-1$ and $t$:
	\begin{equation*}
	X_t = \begin{cases}
	X_{t-1} - N^{t-1}_t, & \text{ if } X_{t-1} - N^{t-1}_t \in R^C_{t}\cup R^S_t, \\
	y_t, & \text{ if } X_{t-1} - N^{t-1}_t \in R^O_{t}.
	\end{cases}
	\end{equation*}
	That is, the inventory at $t$ is equal to inventory at $t-1$ minus the demand, if we stay in the continuation or stopping region; and the inventory at $t$ is equal to the order-up-to level if we enter the ordering region. Then, since $y_t\notin R^S_t$, the following relations hold:
	\begin{align}
	X_{t-1} - (N_t - N_{t-1}) \notin R^S_t \iff &  X_t \notin R^S_t, \label{Eq: dont enter R^S} \\ 
	X_{t-1} - N^{t-1}_t \in R^C_t \implies & X_t\in R^C_t, \label{Eq: enter R^C} \\
	X_{t-1} - N^{t-1}_t \in R^O_t \implies & X_t = y_t. \label{Eq: enter R^O}
	\end{align}
	Next, let us combine the events that occur after $t$ since they will be fixed throughout the proof:
	\begin{equation*}
	E := \cb{X_{t+1}\notin R^S_{t+1}, \dots,\, X_{m-1}\notin R^S_{m-1},\, X_m \in R^S_m}.
	\end{equation*}
	Then, 
	\begin{align*}
	P(t-1,x) = & \PP\cb{\condition{  X_{t} \notin R^S_{t},\dots,\, X_{m-1}\notin R^S_{m-1},\, X_m \in R^S_m } X_{t-1} = x } \\
	= & \PP \of{\condition{ E \cap \cb{X_t \notin R^S_t} } X_{t-1}=x } \\
	= & \PP\of{\condition{ E \cap \cb{ X_{t-1} - N^{t-1}_t \notin R^S_t } } X_{t-1}=x  } \quad \text{(Relation \eqref{Eq: dont enter R^S})} \\
	= & \PP\of{\condition{ E \cap \cb{ x - N^{t-1}_t \notin R^S_t } } X_{t-1}=x  } \\
	= & \PP\of{\condition{ E\cap \of{\cb{x-N^{t-1}_t \in R^O_t} \cup \cb{x-N^{t-1}_t \in R^C_t} }  } X_{t-1}=x } \quad \text{(Since $R^S_t,R^O_t,R^C_t$ are disjoint)}
	\end{align*}
	Moreover, notice that the sets $\cb{x - N^{t-1}_t \in R^C_t }$ and $\cb{x - N^{t-1}_t \in R^O_t}$ are disjoint since 
	\begin{equation*}
	\cb{x - N^{t-1}_t \in R^C_t } \cap \cb{x - N^{t-1}_t \in R^O_t} = \cb{x - N^{t-1}_t \in R^C_t \cap R^O_t} = \cb{x - N^{t-1}_t \in \emptyset} = \emptyset,
	\end{equation*}
	as $R^C_t$ and $R^O_t$ are disjoint. Therefore, after applying Lemma \ref{Lem: P(a,b,c|d) = P(a,b|d)+P(a,c|d)}, we get
	\begin{align*}
	P(t-1,x) = & \PP\of{\condition{ E \cap \of{\cb{x - N^{t-1}_t \in R^C_t }\cup \cb{x - N^{t-1}_t \in R^O_t} } } X_{t-1}=x } \\
	= & \PP\of{\condition{E\cap \cb{x-N^{t-1}_t\in R^C_t} }X_{t-1}=x } + \PP\of{\condition{ E\cap \cb{x-N^{t-1}_t \in R^O_t} } X_{t-1}=x } \\
	= & \sum_{\substack{n\in\ZZ_+: \\ x-n\in R^C_{t}}} \PP\of{\condition{E\cap \cb{N^{t-1}_t = n} }X_{t-1}=x } + \sum_{\substack{ n\in\ZZ_+ :\\ x-n \in R^O_t}} \PP\of{\condition{E\cap \cb{N^{t-1}_t =n} }X_{t-1}=x}. 
	\end{align*}
	Moreover, the first summand is
	\begin{align*}
	&\PP\of{\condition{E\cap \cb{N^{t-1}_t = n} }X_{t-1}=x } \\
	& = \frac{1}{\PP\{X_{t-1}=x\}} \PP\of{E \cap\cb{N^{t-1}_t=n}\cap \cb{X_{t-1}=x}} \\
	& = \frac{1}{\PP\{X_{t-1}=x\}} \PP\of{E \cap\cb{N^{t-1}_t=n}\cap \cb{X_{t-1}=x} \cap \cb{X_t=x-n}} \quad\text{(By relation \eqref{Eq: enter R^C} and $x-n\in R^C_t$)} \\
	& = \frac{1}{\PP\{X_{t-1}=x\}} \PP\of{E | N^{t-1}_t=n, X_{t-1}=x, X_t = x-n} \PP\cb{N^{t-1}_t=n, X_{t-1}=x, X_t=x-n} \\
	& = \frac{1}{\PP\{X_{t-1}=x\}} \PP\of{E | X_t =x-n} \PP\cb{N^{t-1}_t=n, X_{t-1}=x, X_t=x-n} \quad \text{(Markov property of $X_t$)} \\
	& = \frac{1}{\PP\{X_{t-1}=x\}} \PP\of{E | X_t =x-n} \PP\cb{N^{t-1}_t=n, X_{t-1}=x} \quad \text{(Relation \eqref{Eq: enter R^C})} \\
	& = \frac{1}{\PP\{X_{t-1}=x\}} \PP\of{E | X_t =x-n} \PP\of{N^{t-1}_t=n} \PP\of{X_{t-1}=x} \quad \text{($N^{t-1}_t$ and $X_{t-1}$ are independent)} \\
	& = \PP(E | X_t = x-n) \PP\{N^{t-1}_t =n\} \\
	& = P(t,x-n) \PP\{N^{t-1}_t =n\} \quad \text{(Definition of $P$)}. 
	\end{align*}
	Furthermore, by the applying same steps, we can express the other summand as
	\begin{align*}
	&\PP\of{\condition{E\cap \cb{N^{t-1}_t =n} }X_{t-1}=x} \\
	& = \frac{1}{\PP\{X_{t-1}=x\}} \PP\of{E\cap \cb{N^{t-1}_t = n} \cap \cb{X_{t-1}=x} \cap \cb{X_t=y_t}} \quad \text{(Since $x-n\in R^O_t$ and relation \eqref{Eq: enter R^O})} \\
	& = \PP\of{E | X_t = y_t} \PP\cb{N^{t-1}_t =n} \\
	& = P(t,y_t) \PP\cb{N^{t-1}_t =n}. 
	\end{align*}
	Hence, the following equalities hold: 
	\begin{align*}
	P(t-1,x) = & \sum_{\substack{n\in\ZZ_+: \\ x-n\in R^C_{t}}} \PP\of{\condition{E\cap \cb{N^{t-1}_t = n} }X_{t-1}=x } + \sum_{\substack{ n\in\ZZ_+ :\\ x-n \in R^O_t}} \PP\of{\condition{E\cap \cb{N^{t-1}_t =n} }X_{t-1}=x} \\
	= & \sum_{\substack{n\in\ZZ_+: \\ x-n\in R^C_{t}}} P(t,x-n) \PP\cb{N^{t-1}_t =n} + \sum_{\substack{ n\in\ZZ_+ :\\ x-n \in R^O_t}} P(t,y_t) \PP\cb{N^{t-1}_t =n}.
	\end{align*}
\end{proof}

\section{Proof of Results in Subsubsection \ref{Subsubsection: Bound on the Switching Time}}
\label{Section: Proof of Results in Subsubsection: Inventory at 0 affects switching time}
In this section, we prove Proposition \ref{Proposition: Upper bound on tau^*} and Proposition \ref{Proposition: Lower bound on tau^*}. For the brevity of notation, we take derivative of integrals when the functions are right-continuous. To ensure the existence, it is possible to take right-derivative and obtain the same expressions. We characterize an upper bound $\tau^{ub}$ and a lower bound $\tau^{lb}$ by stating, respectively, that $\frac{\partial \cC(x,\tau)}{\partial \tau}\geq 0$ for every $\tau \geq \tau^{\text{ub}}$ and $\frac{\partial \cC(x,\tau)}{\partial \tau} \leq 0$ for every $\tau\leq \tau^{\text{lb}}$. Lemma \ref{Lemma: tau_ub} shows a condition which makes the first derivative of $\cC(x,\tau)$ positive. 

\begin{lem} \label{Lemma: tau_ub}
	Let $x\in\ZZ_+$ and $\tau\in[0,T]$. If
	\[
	\mathbb{P}(N_\tau \geq x) \tilde{c}_2(\tau) \geq \mathbb{P}(N_\tau \leq x-1) [c_{3}(\tau) + c_{4}  ],
	\]
	then $ \frac{\partial \cC(x,\tau)}{\partial \tau} \geq 0.$
\end{lem}
\begin{proof}
	By proceeding as in \cite{frenk2019exact}, we can express $\cC(x,\tau)$ as
	\begin{align*}
	\cC(x,\tau) 
	= & c_{4}x + \EE \bigg(\int_{0}^{\tau} e^{-\delta u} \lambda(u) [- c_{4} - c_{2}(u)] \PP \cb{N_u \leq x-1} du  \bigg) 
	+ \int_{0}^{\tau} e^{-\delta u} \lambda(u) \tilde{c}_2(u)du \notag \\
	& + (c_{1}-\delta c_{4}) \int_{0}^{\tau} e^{-\delta u} \EE [(x-N_u)^+] du
	+ \int_{0}^{T} e^{-\delta u} c_{3}(u) \lambda(u)du.
	\end{align*}
	Therefore, taking derivative with respect to $\tau$ yields
	\begin{align*}
	& \frac{\partial \cC(x,\tau)}{\partial \tau}\\  
	& = e^{-\delta \tau} \lambda(\tau) [- c_{4} - c_{2}(\tau)] \PP \cb{N_\tau \leq x-1}  + e^{-\delta \tau} \lambda(\tau) \tilde{c}_2(\tau)  + e^{-\delta \tau}(c_{1}-\delta c_{4}) \EE[(x-N_\tau)^+] \\
	& = e^{-\delta \tau} \lambda(\tau) [- c_{4} - c_{3}(\tau) - \tilde{c}_2(\tau)] \PP \cb{N_\tau \leq x-1}  + e^{-\delta \tau} \lambda(\tau) \tilde{c}_2(\tau)  + e^{-\delta \tau}(c_{1}-\delta c_{4}) \sum_{k=0}^{x-1} \PP\cb{N_\tau \leq k} \\
	& = e^{-\delta \tau} \lambda(\tau) [- c_{4} - c_{3}(\tau)] \PP \cb{N_\tau \leq x-1} + e^{-\delta \tau} \lambda(\tau) \tilde{c}_2(\tau) \PP\cb{N_\tau\geq x}  + e^{-\delta \tau}(c_{1}-\delta c_{4}) \sum_{k=0}^{x-1} \PP\cb{N_\tau \leq k}  \\
	& = e^{-\delta \tau} \lambda(\tau) \Big( \PP\cb{N_\tau\geq x} \tilde{c}_2(\tau) - \PP \cb{N_\tau \leq x-1} [c_{3}(\tau) + c_{4}   ] \Big)  + e^{-\delta \tau}(c_{1}-\delta c_{4}) \sum_{k=0}^{x-1} \PP\cb{N_\tau \leq k},
	\end{align*}
	where the second equality uses Lemma \ref{Lemma: E[x-N(tau)^+] = sum P(N(tau) <= k)}.
\end{proof} 

\begin{proof}[Proof of \autoref{Proposition: Upper bound on tau^*}]
	We first note that $\PP\cb{N_\tau \geq x} = 1 - \PP\cb{N_\tau \leq x-1}$, so
	\begin{align*}
	& \PP\cb{N_\tau\geq x} \tilde{c}_2(\tau) \geq \PP \cb{N_\tau \leq x-1} [c_{3}(\tau) + c_{4}  ] \\
	\iff & \tilde{c}_2(\tau) \geq \PP\cb{N_\tau\leq x-1} [c_{3}(\tau) + \tilde{c}_2(\tau) + c_{4}  ] \\
	\impliedby & \tilde{c}_2(T) \geq \PP\cb{N_\tau\leq x-1} [c_{3}(\tau) + \tilde{c}_2(\tau) + c_{4}  ],
	\end{align*}
	where last implication is due to $\tilde{c}_2(.)$ being non-increasing. Since $\tau^{\text{ub}}$ satisfies inequality \eqref{Equation: Upper bound on tau^*}, it follows from Lemma \ref{Lemma: tau_ub} that
	\[
	\left.\frac{\partial \cC(x,\tau)}{\partial\tau} \right\vert_{\tau=\tau^{\text{ub}}} \geq 0.
	\]
	Therefore, it suffices to show that for any $\tau\geq\tau^{ub}$, $\frac{\partial \cC(x,\tau)}{\partial\tau} \geq 0$. To see this, we show that the function 
	$\tau\mapsto f(\tau) := \PP \cb{N_\tau \leq x-1} [c_{2}(\tau) + c_{4}]$
	is non-increasing. Taking the derivative of $f$ yields that
	\[
	\frac{\partial f(\tau)}{\partial \tau}  = -\lambda(\tau) \PP\cb{N_\tau = x-1} [c_{2}(\tau) + c_{4}] + \PP \cb{N_\tau \leq x-1} c_{2}^{\prime}(\tau).
	\]
	The first term is negative since it is assumed that $c_{2}(\tau) + c_{4}  \geq 0$ and that $\lambda(\tau) \geq 0$. The second term is negative since $\tilde{c}_2(\tau)$ is non-increasing. Therefore, $f$ is non-increasing.
\end{proof}

\begin{lem} \label{Lemma: for lower bound on tau^*}
	If $\lambda(\tau) \geq 1$ and
	\begin{equation*}
	\PP \cb{N_\tau \leq x-1} [c_{2}(\tau) + c_{4}  ] \geq x(c_{1}-\delta c_{4}) + \tilde{c}_2(0), 
	\end{equation*}
	then $\frac{\partial \cC(x,\tau)}{\partial \tau} \leq 0$.
\end{lem}
\begin{proof}
	Using the same steps in Lemma \ref{Lemma: tau_ub} yields 
	\begin{align*}
	\frac{\partial \cC(x,\tau)}{\partial \tau }
	= & e^{-\delta \tau} \lambda(\tau) \tilde{c}_2(\tau) + e^{-\delta \tau}(c_{1}-\delta c_{4}) \sum_{k=0}^{x-1} \PP\cb{N_\tau\leq k} - e^{-\delta \tau} \lambda(\tau) [c_{2}(\tau) + c_{4}] \PP \cb{N_\tau \leq x-1}.
	\end{align*}
	Moreover, it is possible to see that
	\begin{equation} \label{Equation: observation in Lemma: for lower bound on tau^*}
	\sum_{k=0}^{x-1} \PP\cb{N_\tau\leq k} \leq x \quad \text{and} \quad \tilde{c}_2(\tau) \leq \tilde{c}_2(0)
	\end{equation}
	since the function $\tilde{c}_2(.)$ is non-increasing. Therefore,
	\begin{align*}
	& \frac{\partial \cC(x,\tau)}{\partial \tau} \leq 0 \\
	& \iff e^{-\delta \tau} \lambda(\tau) \tilde{c}_2(\tau) + e^{-\delta \tau}(c_{1}-\delta c_{4}) \sum_{k=0}^{x-1} \PP\cb{N_\tau\leq k}  \leq e^{-\delta \tau} \lambda(\tau) [c_{2}(\tau) + c_{4}] \PP \cb{N_\tau \leq x-1} \\
	& \impliedby \lambda(\tau) \tilde{c}_2(0) + (c_{1}- \delta c_{4})x  \leq \lambda(\tau) [c_{2}(\tau) + c_{4}] \PP \cb{N_\tau \leq x-1} \quad\text{(by \eqref{Equation: observation in Lemma: for lower bound on tau^*})} \\
	& \impliedby \lambda(\tau) \tilde{c}_2(0) + \lambda(\tau) (c_{1}- \delta c_{4})x  \leq \lambda(\tau) [c_{2}(\tau) + c_{4}] \PP \cb{N_\tau \leq x-1} \quad(\text{since }\lambda\geq 1) \\
	& \iff \tilde{c}_2(0) + (c_{1}-\delta c_{4}) x  \leq [c_{2}(\tau) + c_{4}  ] \PP \cb{N_\tau \leq x-1}.
	\end{align*}
\end{proof}

\begin{proof}[Proof of \autoref{Proposition: Lower bound on tau^*}]
	It suffices to show that any $\tau\in[0,\tau^{lb}]$ satisfy the inequality in \eqref{Equation: Lower bound on tau^*}, so that $\frac{\partial \cC(x,\tau)}{\partial \tau} \leq 0$ due to Lemma \ref{Lemma: for lower bound on tau^*}. To achieve this, we first note that $\lambda$ is non-increasing, therefore, any $\tau\leq\tau^{lb}$ satisfies $\lambda(\tau)\geq 1$. Next, we show that the function
	\begin{equation*}
	\tau\mapsto g(\tau) := \PP \cb{N_\tau \leq x-1} [c_{2}(\tau) + c_{4}  ]
	\end{equation*}
	is a non-increasing function. Taking derivative of $g$ yields that
	\[
	\frac{\partial g(\tau)}{\partial \tau} = -\lambda(\tau) \PP \cb{N_\tau \leq x-1} [c_{2}(\tau) + c_{4}  ]  + \PP \cb{N_\tau \leq x-1} c_{2}^{\prime}(\tau)
	\]
	The first term is negative due to the assumptions that $c_{2}(\tau) + c_{4} \geq 0$ and that $\lambda(\tau) \geq 0$. The second term is negative since both $c_{3}$ and $\tilde{c}_2$ is non-increasing.
\end{proof}

\section{Proof of \autoref{Proposition: S is increasing in tau}}
\label{Section: Proof of Results in Subsubsection S is increasing in tau}

In this section, we prove two lemmata and Proposition \ref{Proposition: S is increasing in tau}. In the sequel, we use the following forms of $\cC(x,\tau)$, $\Delta_{x} \cC(x, \tau):=\cC(x+1, \tau)-\cC(x, \tau)$ and $\Delta_{x}^{2} \cC(x, \tau):=\Delta_{x} \cC(x+1, \tau)-\Delta_{x} \cC(x, \tau)$ shown by \cite{frenk2019exact} in relations (4)-(6), (14):
\begin{flalign} \label{Equation: C(x,tau)}
\cC(x,\tau) 
= & c_{4}x + \EE \bigg(\int_{0}^{\tau} e^{-\delta u} \lambda(u) [- c_{4} - c_{2}(u)] \PP \cb{N_u \leq x-1} du  \bigg) 
 + \int_{0}^{\tau} e^{-\delta u} \lambda(u) \tilde{c}_2(u)du \notag \\
& + (c_{1}-\delta c_{4}) \int_{0}^{\tau} e^{-\delta u} \EE [(x-N_u)^+] du
 + \int_{0}^{T} e^{-\delta u} c_{3}(u) \lambda(u)du,\\
\cC(0,\tau) = & \int_{0}^{\tau} e^{-\delta u} \lambda(u) \tilde{c}_2(u) du + \int_{0}^{T} e^{-\delta u} c_{3}(u) \lambda(u) du, \label{Equation: C(0,tau)} \\
\Delta_x \cC(x,\tau) 
= & c_{4} + \int_{0}^{\tau} e^{-\delta u} \lambda(u) [- c_{4} - c_{2}(u)] \PP\cb{N_u = x} du \notag \\
& + (c_{1}-\delta c_{4}) \int_{0}^{\tau} e^{-\delta u} \PP\cb{N_u \leq x} du, \label{Equation: Delta_xC(x,tau)} \\
\Delta_x^2\cC(x-1,\tau) 
= & e^{-\delta \tau} \big(c_{2}(\tau) + c_{4} \big) \PP\cb{N_\tau = x} 
+ \int_{0}^{\tau} e^{-\delta u} \big[ c_{1}-c_{2}^{\prime}(u) + \delta c_{2}(u) \big] \PP\cb{N_u = x} du \notag \\
& - \sum_{i\leq m,\, l_i\leq \tau} e^{-\delta l_i } \Delta c_{2}(l_i) \PP\cb{N_{l_i} = x}. \label{Equation: Delta_x2C(x,tau)}
\end{flalign}

Lemma \ref{Lemma: Delta_xC(0,tau) is non-increasing in tau}, Lemma \ref{Lemma: Second derivative of C(x,tau) is non-increasing in tau} and Proposition \ref{Proposition: S is increasing in tau} essentially utilize the idea that $\Delta_x \cC(x,\tau)$ is a decreasing function of $\tau$ and an increasing function of $x$ under the conditions of Proposition \ref{Proposition: S is increasing in tau}. Therefore, if $\tau$ increases, then $S(\tau)$ increases as well since it is the minimum $x$ value satisfying the first order condition, namely
\begin{equation} \label{Equation: S_tau}
S(\tau) = \min\cb{x\in\ZZ_+: \bar{c} + \Delta_x \cC(x,\tau)\geq 0}.
\end{equation}

\begin{lem} \label{Lemma: Delta_xC(0,tau) is non-increasing in tau}
	For every $\epsilon\in[0,T]$ and every $\tau\in[0,T]$ such that
	\[
	c_{1} \leq \lambda(\tau+\epsilon)c_{3}(\tau+\epsilon), 
	\]
	we have $\Delta_x \cC(0,\tau+\epsilon) < \Delta_x \cC(0,\tau)$.
\end{lem} 
\begin{proof} 
	By using the expression for $\Delta_x\cC(x,\tau)$ in \eqref{Equation: Delta_xC(x,tau)}, we obtain
	\begin{align*}
	& \Delta_x \cC(0,\tau+\epsilon) - \Delta_x \cC(0,\tau) \\
	 &= \int_{\tau}^{\tau+\epsilon} e^{-\delta u}\lambda(u) \sqb{-c_{4} - c_{2}(u)} \PP\cb{N_u=0} du
	  (c_{1}-\delta c_{4}) \int_{\tau}^{\tau+\epsilon} e^{-\delta u} \PP\cb{N_u= 0} du \\
	 &= \int_{\tau}^{\tau+\epsilon} e^{-\delta u}\PP\cb{N_u=0} \bigg( c_{1}-\delta c_{4} + \lambda(u) [- c_{4} - c_{3}(u) - \tilde{c}_2(u) ] \bigg) du\\
	 &= \int_{\tau}^{\tau+\epsilon} e^{-\delta u}\PP\cb{N_u=0}\negthinspace \Big(\negthinspace -\delta c_{4} + \lambda(u) [ \underbrace{- c_{4}  - \tilde{c}_2(u)}_{< 0} ] \Big) du
	 \negthinspace +\negthinspace \int_{\tau}^{\tau+\epsilon} e^{-\delta u}\PP\cb{N_u=0} \negthinspace\Big( \negthinspace \underbrace{c_{1}- \lambda(u) c_{3}(u)}_{\leq 0} \Big) du\\
	&<  0,
	\end{align*}
	where the inequality $- c_{4} - \tilde{c}_2(u)<0$ holds for every $u\in[\tau,\tau+\epsilon]$ since $c_{4}$ and $\tilde{c}_2(u)$ are positive. Moreover, the inequality $c_{1}-\lambda(u)c_{3}(u)\leq 0$ holds for every $u\in[\tau,\tau+\epsilon]$ since
	\[
	c_{1} \leq \lambda(\tau+\epsilon)c_{3}(\tau+\epsilon) 
	\leq  \lambda(u)c_{3}(u),
	\]
	where first inequality is due to the condition of the lemma and the second inequality is because $\lambda$ and $c_{3}$ are non-increasing.
\end{proof}

The next lemma is helpful while stating in Proposition \ref{Proposition: S is increasing in tau} that if expected total demand exceeds the order amount and cost rate of outside source does not decline sufficiently, then the order amount should increase.
\begin{lem} \label{Lemma: Second derivative of C(x,tau) is non-increasing in tau}
	For every $x\in\ZZ_+$, every $\epsilon\in[0,T]$ and every $\tau_2\in[0,T]$ such that
	\begin{align*}
	& (i)\, x < \Lambda(\tau_2), 
	& (ii)\, c_{1} \leq \sqb{\frac{\Lambda(u) - x}{\Lambda(u)}} \lambda(u)c_{3}(u) \text{ for all } u\in[\tau_2,\tau_2+\epsilon],
	\end{align*}
	we have 
	\begin{equation} \label{Equation: Delta_x^2 C(x,tau+epsilon) < Delta_x^2C(x,tau)}
	\Delta_x^2\cC(x-1,\tau_2+\epsilon) < \Delta_x^2\cC(x-1,\tau_2).
	\end{equation}
\end{lem}
\begin{proof}
	For the non-homogeneous Poisson process $N$ with right-continuous intensity function $\lambda$, the right-directional derivative of the function $u\mapsto\psi(u) = \PP\cb{N_u=x}$ exists and it is given by
	\begin{align*}
	\psi^\prime(u+) := & \lim_{\epsilon\downarrow 0} \frac{1}{\epsilon} \sqb{\psi(u+\epsilon) - \psi(u)} \\
	= & -\lambda(u) e^{-\Lambda (u)} \frac{\Lambda(u)^x}{x!} + e^{-\Lambda(u)} \frac{\Lambda(u)^{x-1}}{(x-1)!}\lambda(u) 
	 = -\lambda(u) \PP\cb{N_u = x} \bigg[1- \frac{x}{\Lambda(u)}\bigg].
	\end{align*}
	Moreover, we observe that the function $\Lambda$ is strictly increasing and
	\begin{equation} \label{Relation: Derive of P(N_u=x) is negative}
	\psi^\prime(u+) < 0 \text{ for every } u\in[0,T] \text{ such that } \Lambda(u) > x.  \\
	\end{equation}
	After applying chain rule to the function
	\[
	\tau\rightarrow \underbrace{e^{-\delta\tau}}_{\circled{1} }  \underbrace{(c_{2}(\tau) +c_{4})}_{\circled{2}} \underbrace{\PP\cb{N_\tau=x}}_{\circled{3}}
	\]
	in \eqref{Equation: Delta_x2C(x,tau)}, we obtain
	\begin{align*}
	& \Delta_x^2\cC(x-1,\tau_2+\epsilon) - \Delta_x^2\cC(x-1,\tau_2) \\
	= & - \int_{\tau_2}^{\tau_2+\epsilon} \underbrace{\delta e^{-\delta u}}_{\circled{1}} \big( c_{2}(u) + c_{4}   \big) \PP\cb{N_u = x} du \\
	& + \int_{\tau_2}^{\tau_2+\epsilon} e^{-\delta u} \underbrace{c_{2}^{\prime}(u)}_{\circled{2}} \PP\cb{N_u=x} du + \sum_{i\leq m,\, l_i\leq \tau} e^{-\delta l_i } \Delta c_{2}(l_i) \PP\{ N_{l_i} = x\} \\
	& - \int_{\tau_2}^{\tau_2+\epsilon} e^{-\delta u} \big( c_{2}(u) + c_{4} \big) \underbrace{\lambda(u) \PP\cb{N_u = x} \bigg[ 1-\frac{x}{\Lambda(u)} \bigg]}_{\circled{3}}  du \\
	& + \int_{\tau_2}^{\tau_2+\epsilon} e^{-\delta u} \big[ c_{1}-c_{2}^{\prime}(u) + \delta c_{2}(u) \big] \PP\cb{N_u = x} du
	  - \sum_{i\leq m,\, \tau_2\leq l_i\leq \tau_2+\epsilon} e^{-\delta l_i } \Delta c_{2}(l_i) \PP\{ N_{l_i} = x\}. 
	\end{align*}
	Notice that all the integrals include the expression $e^{-\delta u}\PP\cb{N_u=x}$. Grouping them gives
	\begin{align*}
	& \Delta_x^2\cC(x-1,\tau_2+\epsilon) - \Delta_x^2\cC(x-1,\tau_2) \\
	= & \int_{\tau_2}^{\tau_2+\epsilon} e^{-\delta u} \PP\cb{N_u = x} \\
	& \quad\times \bigg[ -\delta \big( \underbrace{c_{2}(u)} + c_{4} \big) + \underbrace{c_{2}^{\prime}(u)} -\lambda(u) \bigg[ 1-\frac{x}{\Lambda(\tau)} \bigg] (c_{2}(u) + c_{4}  ) \\
	& \quad\quad + c_{1}- \underbrace{c_{2}^{\prime}(u)} + \underbrace{\delta c_{2}(u)})  \bigg] du \quad\text{ (underbraced terms cancel each other)}\\
	= & \int_{\tau_2}^{\tau_2+\epsilon} e^{-\delta u} \PP\cb{N_u = x} \bigg[ (c_{1}-\delta c_{4}) - \lambda(u) \bigg[\frac{\Lambda(u)-x}{\Lambda(u)}\bigg] (c_{2}(u) + c_{4}  ) \bigg] du.
	\end{align*}
	Next, after separating the remaining terms, we can see that
	\begin{align}
	& \Delta_x^2\cC(x-1,\tau_2+\epsilon) - \Delta_x^2\cC(x-1,\tau_2) \notag \\
	= & \int_{\tau_2}^{\tau_2+\epsilon} e^{-\delta u} \PP\cb{N_u = x} (-\delta c_{4}) du \notag \\
	& + \int_{\tau_2}^{\tau_2+\epsilon} e^{-\delta u} \PP\cb{N_u = x} \bigg[- \underbrace{\lambda(u)}_{\geq \lambda(\tau_2+\epsilon) } \bigg[\frac{\Lambda(u)-x}{\Lambda(u)}\bigg] c_{4} \bigg] du \quad \text{ ($\lambda$ is non-increasing) } \notag \\
	& + \int_{\tau_2}^{\tau_2+\epsilon} e^{-\delta u} \PP\cb{N_u = x}\sqb{- \underbrace{\lambda(u)}_{\geq \lambda(\tau_2+\epsilon)} \sqb{\frac{\Lambda(u)-x}{\Lambda(u)}} (\underbrace{\tilde{c}_2(u)  }_{\geq \tilde{c}_2(T) } ) } du \quad \text{ ($\tilde{c}_2$ is non-increasing)} \notag \\
	& + \int_{\tau_2}^{\tau_2+\epsilon} e^{-\delta u} \underbrace{\PP\cb{N_u = x}}_{=\psi(u)\geq \psi(\tau_2+\epsilon)} \sqb{c_{1}-\lambda(u) \sqb{\frac{\Lambda(u)-x}{\Lambda(u)}} c_{3}(u)} du \notag \\
	& \quad\text{ (Relation \eqref{Relation: Derive of P(N_u=x) is negative} and condition $(i)$)} \notag \\ 
	\leq & - \int_{\tau_2}^{\tau_2+\epsilon} e^{-\delta u} \PP\cb{N_u = x} (\delta c_{4}) du \notag\\
	& - \lambda(\tau_2+\epsilon)c_{4} \int_{\tau_2}^{\tau_2+\epsilon} e^{-\delta u} \PP\cb{N_u = x} \bigg[\frac{\Lambda(u)-x}{\Lambda(u)}\bigg] du \notag\\
	& - \lambda(\tau_2+\epsilon)(\tilde{c}_2(T) ) \int_{\tau_2}^{\tau_2+\epsilon} e^{-\delta u} \PP\cb{N_u = x} \sqb{\frac{\Lambda(u)-x}{\Lambda(u)}} du \notag\\
	& + e^{-\delta (\tau_2+\epsilon)} \PP\cb{N_{\tau_2+\epsilon} = x} \int_{\tau_2}^{\tau_2+\epsilon}  \sqb{c_{1}-\lambda(u) \sqb{\frac{\Lambda(u)-x}{\Lambda(u)}} c_{3}(u)} du \label{Relation: Last inequality after chain rule} \\
	< & 0. \notag
	\end{align}
	In \eqref{Relation: Last inequality after chain rule}, the first and second terms are negative due to the assumption $c_{4}\in\RR_+$ and condition $(i)$ of the lemma. The third term is negative since $\tilde{c}_2(T)\geq 0$ and condition $(i)$ of the lemma. The last term is negative due to condition $(ii)$ of the lemma. This concludes the proof. 
\end{proof}

\begin{proof}[Proof of \autoref{Proposition: S is increasing in tau}]
	The function $x \mapsto \cC(x,\tau)$ being discrete-convex implies that the function $x \mapsto \Delta_x \cC(x,\tau)$ is non-decreasing. Moreover, by the definition of $S(\tau)$, the first order condition in equation \eqref{Equation: S_tau} has to be satisfied by $S(\tau_2)$ and $\tau_2$ as well as $S(\tau_2+\epsilon)$ and $\tau_2+\epsilon$, meaning that
	\[
	\bar{c} + \Delta_x\cC(S(\tau_2),\tau_2) \geq 0 \text{ and } \bar{c} + \Delta_x \cC(S(\tau_2+\epsilon),\tau_2+\epsilon) \geq 0.
	\]
	If we can show that 
	\begin{equation} \label{Equation: Delta_x C(S(tau_2),tau_2 epsilon) less C(S(tau_2),tau_2)}
	\Delta_x \cC(S(\tau_2), \tau_2+\epsilon) < \Delta_x \cC(S(\tau_2),\tau_2),
	\end{equation}
	then $S(\tau_2)\leq S(\tau_2+\epsilon)$ must hold. To show \eqref{Equation: Delta_x C(S(tau_2),tau_2 epsilon) less C(S(tau_2),tau_2)}, we proceed in three steps. First,  condition $(iv)$ implies that
	\[
	c_{1} \underbrace{\leq}_{(iv)} \sqb{\frac{\Lambda(\tau_2+\epsilon)-S(\tau_1)}{\Lambda(\tau_2+\epsilon)} \lambda(\tau_2+\epsilon) c_{3}(\tau_2+\epsilon) } \leq \lambda(\tau_2+\epsilon) c_{3}(\tau_2+\epsilon).
	\]
	By using Lemma \ref{Lemma: Delta_xC(0,tau) is non-increasing in tau}, we obtain
	\[
	\Delta_x\cC(0,\tau_2+\epsilon) < \Delta_x\cC(0,\tau_2).
	\] 
	Next, observe from conditions $(iv)$ and $(ii)$ that for every $u\in[\tau_2,\tau_2+\epsilon]$,
	\[
	c_{1} \underbrace{\leq}_{(iv)} \sqb{\frac{\Lambda(u) - S(\tau_1)}{\Lambda(u)}} \lambda(u)c_{3}(u) 
	\underbrace{\leq}_{(ii)} \sqb{\frac{\Lambda(u) - S(\tau_2)}{\Lambda(u)}} \lambda(u)c_{3}(u).
	\]
	By using Lemma \ref{Lemma: Second derivative of C(x,tau) is non-increasing in tau}, we obtain
	\[
	\Delta_x^2 \cC(S(\tau_2)-1,\tau_2+\epsilon) < \Delta_x^2 \cC(S(\tau_2)-1,\tau_2)
	\]
	and similarly, for all $x\in\cb{1,2,\dots,S(\tau_2)-1}$, we have
	\[
	\Delta_x^2 \cC(x-1,\tau_2+\epsilon) < \Delta_x^2 \cC(x-1,\tau_2).
	\]
	Finally, we obtain
	\begin{align*}
	\Delta_x \cC(S(\tau_2), \tau_2+\epsilon) 
	= & \Delta_x \cC(0,\tau_2 + \epsilon) + \sum_{x=0}^{S(\tau_2)-1} \Delta_x^2 \cC(x,\tau_2+\epsilon) \\
	< & \Delta_x \cC(0,\tau_2) + \sum_{x=0}^{S(\tau_2)-1} \Delta_x^2 \cC(x,\tau_2)
	=\Delta_x \cC(S(\tau_2),\tau_2).
	\end{align*}
\end{proof}

\section{Numbering for the Parameter Settings}\label{Section: Number of Parameter Settings}
This section assigns a number for each parameter setting where the parameters take values in the sets presented in Table \ref{Table: Our numerical Values}. The Table \ref{Table: Numbering for Parameter Settings} below show the assigned numbers. Also recall that $\bar{c}=100$, $c_1=0.01\bar{c}$ and $\bar{c}_3=2\bar{c}$. We present the setup cost $K$ and initial inventory $x$ values alongside the related result.

\begin{table}[h!]
	\resizebox{\textwidth}{!}{
	\begin{tabular}{|ccccccc|ccccccc|ccccccc|}
		\hline
		\# & $\lambda$ & $T$   & $c_4$ & $\gamma$    & $\delta$   & $\bar{c}_2$ & \# & $\lambda$ & $T$   & $c_4$ & $\gamma$    & $\delta$   & $\bar{c}_2$ & \# & $\lambda$ & $T$   & $c_4$ & $\gamma$    & $\delta$   & $\bar{c}_2$ \\ \hline
		1  & Conv   & 50  & 25   & 0.01     & 0.005    & 200                         & 43 & Conc   & 50  & -25  & 0.01     & $10^{-6}$ & 200                         & 85  & Lin    & 100 & 25   & $10^{-6}$ & 0.005    & 200                         \\
		2  & Conv   & 50  & 25   & 0.01     & 0.005    & 1000                        & 44 & Conc   & 50  & -25  & 0.01     & $10^{-6}$ & 1000                        & 86  & Lin    & 100 & 25   & $10^{-6}$ & 0.005    & 1000                        \\
		3  & Conv   & 50  & 25   & 0.01     & $10^{-6}$ & 200                         & 45 & Conc   & 50  & -25  & $10^{-6}$ & 0.005    & 200                         & 87  & Lin    & 100 & 25   & $10^{-6}$ & $10^{-6}$ & 200                         \\
		4  & Conv   & 50  & 25   & 0.01     & $10^{-6}$ & 1000                        & 46 & Conc   & 50  & -25  & $10^{-6}$ & 0.005    & 1000                        & 88  & Lin    & 100 & 25   & $10^{-6}$ & $10^{-6}$ & 1000                        \\
		5  & Conv   & 50  & 25   & $10^{-6}$ & 0.005    & 200                         & 47 & Conc   & 50  & -25  & $10^{-6}$ & $10^{-6}$ & 200                         & 89  & Lin    & 100 & -25  & 0.01     & 0.005    & 200                         \\
		6  & Conv   & 50  & 25   & $10^{-6}$ & 0.005    & 1000                        & 48 & Conc   & 50  & -25  & $10^{-6}$ & $10^{-6}$ & 1000                        & 90  & Lin    & 100 & -25  & 0.01     & 0.005    & 1000                        \\
		7  & Conv   & 50  & 25   & $10^{-6}$ & $10^{-6}$ & 200                         & 49 & Conc   & 100 & 25   & 0.01     & 0.005    & 200                         & 91  & Lin    & 100 & -25  & 0.01     & $10^{-6}$ & 200                         \\
		8  & Conv   & 50  & 25   & $10^{-6}$ & $10^{-6}$ & 1000                        & 50 & Conc   & 100 & 25   & 0.01     & 0.005    & 1000                        & 92  & Lin    & 100 & -25  & 0.01     & $10^{-6}$ & 1000                        \\
		9  & Conv   & 50  & -25  & 0.01     & 0.005    & 200                         & 51 & Conc   & 100 & 25   & 0.01     & $10^{-6}$ & 200                         & 93  & Lin    & 100 & -25  & $10^{-6}$ & 0.005    & 200                         \\
		10 & Conv   & 50  & -25  & 0.01     & 0.005    & 1000                        & 52 & Conc   & 100 & 25   & 0.01     & $10^{-6}$ & 1000                        & 94  & Lin    & 100 & -25  & $10^{-6}$ & 0.005    & 1000                        \\
		11 & Conv   & 50  & -25  & 0.01     & $10^{-6}$ & 200                         & 53 & Conc   & 100 & 25   & $10^{-6}$ & 0.005    & 200                         & 95  & Lin    & 100 & -25  & $10^{-6}$ & $10^{-6}$ & 200                         \\
		12 & Conv   & 50  & -25  & 0.01     & $10^{-6}$ & 1000                        & 54 & Conc   & 100 & 25   & $10^{-6}$ & 0.005    & 1000                        & 96  & Lin    & 100 & -25  & $10^{-6}$ & $10^{-6}$ & 1000                        \\
		13 & Conv   & 50  & -25  & $10^{-6}$ & 0.005    & 200                         & 55 & Conc   & 100 & 25   & $10^{-6}$ & $10^{-6}$ & 200                         & 97  & Cons   & 50  & 25   & 0.01     & 0.005    & 200                         \\
		14 & Conv   & 50  & -25  & $10^{-6}$ & 0.005    & 1000                        & 56 & Conc   & 100 & 25   & $10^{-6}$ & $10^{-6}$ & 1000                        & 98  & Cons   & 50  & 25   & 0.01     & 0.005    & 1000                        \\
		15 & Conv   & 50  & -25  & $10^{-6}$ & $10^{-6}$ & 200                         & 57 & Conc   & 100 & -25  & 0.01     & 0.005    & 200                         & 99  & Cons   & 50  & 25   & 0.01     & $10^{-6}$ & 200                         \\
		16 & Conv   & 50  & -25  & $10^{-6}$ & $10^{-6}$ & 1000                        & 58 & Conc   & 100 & -25  & 0.01     & 0.005    & 1000                        & 100 & Cons   & 50  & 25   & 0.01     & $10^{-6}$ & 1000                        \\
		17 & Conv   & 100 & 25   & 0.01     & 0.005    & 200                         & 59 & Conc   & 100 & -25  & 0.01     & $10^{-6}$ & 200                         & 101 & Cons   & 50  & 25   & $10^{-6}$ & 0.005    & 200                         \\
		18 & Conv   & 100 & 25   & 0.01     & 0.005    & 1000                        & 60 & Conc   & 100 & -25  & 0.01     & $10^{-6}$ & 1000                        & 102 & Cons   & 50  & 25   & $10^{-6}$ & 0.005    & 1000                        \\
		19 & Conv   & 100 & 25   & 0.01     & $10^{-6}$ & 200                         & 61 & Conc   & 100 & -25  & $10^{-6}$ & 0.005    & 200                         & 103 & Cons   & 50  & 25   & $10^{-6}$ & $10^{-6}$ & 200                         \\
		20 & Conv   & 100 & 25   & 0.01     & $10^{-6}$ & 1000                        & 62 & Conc   & 100 & -25  & $10^{-6}$ & 0.005    & 1000                        & 104 & Cons   & 50  & 25   & $10^{-6}$ & $10^{-6}$ & 1000                        \\
		21 & Conv   & 100 & 25   & $10^{-6}$ & 0.005    & 200                         & 63 & Conc   & 100 & -25  & $10^{-6}$ & $10^{-6}$ & 200                         & 105 & Cons   & 50  & -25  & 0.01     & 0.005    & 200                         \\
		22 & Conv   & 100 & 25   & $10^{-6}$ & 0.005    & 1000                        & 64 & Conc   & 100 & -25  & $10^{-6}$ & $10^{-6}$ & 1000                        & 106 & Cons   & 50  & -25  & 0.01     & 0.005    & 1000                        \\
		23 & Conv   & 100 & 25   & $10^{-6}$ & $10^{-6}$ & 200                         & 65 & Lin    & 50  & 25   & 0.01     & 0.005    & 200                         & 107 & Cons   & 50  & -25  & 0.01     & $10^{-6}$ & 200                         \\
		24 & Conv   & 100 & 25   & $10^{-6}$ & $10^{-6}$ & 1000                        & 66 & Lin    & 50  & 25   & 0.01     & 0.005    & 1000                        & 108 & Cons   & 50  & -25  & 0.01     & $10^{-6}$ & 1000                        \\
		25 & Conv   & 100 & -25  & 0.01     & 0.005    & 200                         & 67 & Lin    & 50  & 25   & 0.01     & $10^{-6}$ & 200                         & 109 & Cons   & 50  & -25  & $10^{-6}$ & 0.005    & 200                         \\
		26 & Conv   & 100 & -25  & 0.01     & 0.005    & 1000                        & 68 & Lin    & 50  & 25   & 0.01     & $10^{-6}$ & 1000                        & 110 & Cons   & 50  & -25  & $10^{-6}$ & 0.005    & 1000                        \\
		27 & Conv   & 100 & -25  & 0.01     & $10^{-6}$ & 200                         & 69 & Lin    & 50  & 25   & $10^{-6}$ & 0.005    & 200                         & 111 & Cons   & 50  & -25  & $10^{-6}$ & $10^{-6}$ & 200                         \\
		28 & Conv   & 100 & -25  & 0.01     & $10^{-6}$ & 1000                        & 70 & Lin    & 50  & 25   & $10^{-6}$ & 0.005    & 1000                        & 112 & Cons   & 50  & -25  & $10^{-6}$ & $10^{-6}$ & 1000                        \\
		29 & Conv   & 100 & -25  & $10^{-6}$ & 0.005    & 200                         & 71 & Lin    & 50  & 25   & $10^{-6}$ & $10^{-6}$ & 200                         & 113 & Cons   & 100 & 25   & 0.01     & 0.005    & 200                         \\
		30 & Conv   & 100 & -25  & $10^{-6}$ & 0.005    & 1000                        & 72 & Lin    & 50  & 25   & $10^{-6}$ & $10^{-6}$ & 1000                        & 114 & Cons   & 100 & 25   & 0.01     & 0.005    & 1000                        \\
		31 & Conv   & 100 & -25  & $10^{-6}$ & $10^{-6}$ & 200                         & 73 & Lin    & 50  & -25  & 0.01     & 0.005    & 200                         & 115 & Cons   & 100 & 25   & 0.01     & $10^{-6}$ & 200                         \\
		32 & Conv   & 100 & -25  & $10^{-6}$ & $10^{-6}$ & 1000                        & 74 & Lin    & 50  & -25  & 0.01     & 0.005    & 1000                        & 116 & Cons   & 100 & 25   & 0.01     & $10^{-6}$ & 1000                        \\
		33 & Conc   & 50  & 25   & 0.01     & 0.005    & 200                         & 75 & Lin    & 50  & -25  & 0.01     & $10^{-6}$ & 200                         & 117 & Cons   & 100 & 25   & $10^{-6}$ & 0.005    & 200                         \\
		34 & Conc   & 50  & 25   & 0.01     & 0.005    & 1000                        & 76 & Lin    & 50  & -25  & 0.01     & $10^{-6}$ & 1000                        & 118 & Cons   & 100 & 25   & $10^{-6}$ & 0.005    & 1000                        \\
		35 & Conc   & 50  & 25   & 0.01     & $10^{-6}$ & 200                         & 77 & Lin    & 50  & -25  & $10^{-6}$ & 0.005    & 200                         & 119 & Cons   & 100 & 25   & $10^{-6}$ & $10^{-6}$ & 200                         \\
		36 & Conc   & 50  & 25   & 0.01     & $10^{-6}$ & 1000                        & 78 & Lin    & 50  & -25  & $10^{-6}$ & 0.005    & 1000                        & 120 & Cons   & 100 & 25   & $10^{-6}$ & $10^{-6}$ & 1000                        \\
		37 & Conc   & 50  & 25   & $10^{-6}$ & 0.005    & 200                         & 79 & Lin    & 50  & -25  & $10^{-6}$ & $10^{-6}$ & 200                         & 121 & Cons   & 100 & -25  & 0.01     & 0.005    & 200                         \\
		38 & Conc   & 50  & 25   & $10^{-6}$ & 0.005    & 1000                        & 80 & Lin    & 50  & -25  & $10^{-6}$ & $10^{-6}$ & 1000                        & 122 & Cons   & 100 & -25  & 0.01     & 0.005    & 1000                        \\
		39 & Conc   & 50  & 25   & $10^{-6}$ & $10^{-6}$ & 200                         & 81 & Lin    & 100 & 25   & 0.01     & 0.005    & 200                         & 123 & Cons   & 100 & -25  & 0.01     & $10^{-6}$ & 200                         \\
		40 & Conc   & 50  & 25   & $10^{-6}$ & $10^{-6}$ & 1000                        & 82 & Lin    & 100 & 25   & 0.01     & 0.005    & 1000                        & 124 & Cons   & 100 & -25  & 0.01     & $10^{-6}$ & 1000                        \\
		41 & Conc   & 50  & -25  & 0.01     & 0.005    & 200                         & 83 & Lin    & 100 & 25   & 0.01     & $10^{-6}$ & 200                         & 125 & Cons   & 100 & -25  & $10^{-6}$ & 0.005    & 200                         \\
		42 & Conc   & 50  & -25  & 0.01     & 0.005    & 1000                        & 84 & Lin    & 100 & 25   & 0.01     & $10^{-6}$ & 1000                        & 126 & Cons   & 100 & -25  & $10^{-6}$ & 0.005    & 1000                        \\
		&        &     &      &          &          &                             &    &        &     &      &          &          &                             & 127 & Cons   & 100 & -25  & $10^{-6}$ & $10^{-6}$ & 200                         \\
		&        &     &      &          &          &                             &    &        &     &      &          &          &                             & 128 & Cons   & 100 & -25  & $10^{-6}$ & $10^{-6}$ & 1000                        \\ \hline
	\end{tabular}}
	\caption{\label{Table: Numbering for Parameter Settings} Numbers for each parameter setting. The columns with the symbol \# show the assigned number. An abbreviation is used when $\lambda$ is convex (Conv), concave (Conv), linear (Lin), or constant (Cons).}
\end{table}

\newpage
\bibliographystyle{named}

\begin{thebibliography}{99}


\bibitem[Bayındır et al.(2007)]{bayindir2007assessing}
Bayındır, Z.P., Erkip, N., Güllü, R., 2007. Assessing the benefits of remanufacturing option under one-way substitution and capacity constraint. Computers \& Operations Research, 34 (2), 487-514.

\bibitem[Behfard et al.(2015)]{behfard2015last}
Behfard, S., van der Heijden, M.C., Al Hanbali, A., Zijm, W.H., 2015. Last time buy and repair decisions for spare parts. European Journal of Operational Research, 244 (2), 498-510.

\bibitem[Behfard et al.(2018)]{behfard2018last}
Behfard, S., Al Hanbali, A., van der Heijden, M.C., Zijm, W.H., 2018. Last time buy and repair decisions for fast moving parts. International Journal of Production Economics, 197, 158-173.

\bibitem[Beyer et al.(2010)]{beyer2010markovian}
Beyer, D., Cheng, F., Sethi, S.P., Taksar, M., 2010. Markovian Demand Inventory Models. New York: Springer.

\bibitem[Bradley and Guerrero(2008)]{bradley2008product}
Bradley, J.R., Guerrero, H.H., 2008. Product design for life‐cycle mismatch. Production and Operations Management, 17 (5), 497-512.

\bibitem[Bradley and Guerrero(2009)]{bradley2009lifetime}
Bradley, J.R., Guerrero, H.H., 2009. Lifetime buy decisions with multiple obsolete parts. Production and Operations Management, 18 (1), 114-126.

\bibitem[Callioni et al.(2005)]{callioni2005inventory}
Callioni, G., de Montgros, X., Slagmulder, R., Van Wassenhove, L.N., Wright, L., 2005. Inventory-driven costs. Harvard Business Review, 83 (3), 135-141.

\bibitem[Cattani and Souza(2003)]{cattani2003good}
Cattani, K.D., Souza, G.C., 2003. Good buy? Delaying end-of-life purchases. European Journal of Operational Research, 146 (1), 216-228.

\bibitem[\c{C}{\i}nlar(2011)]{ccinlar2011probability}
Çınlar, E., 2011. Probability and Stochastics. Springer Science \& Business Media.

\bibitem[David et al.(1997)]{david1997dynamic}
David, I., Greenshtein, E., Mehrez, A., 1997. A dynamic‐programming approach to continuous‐review obsolescent inventory problems. Naval Research Logistics (NRL), 44 (8), 757-774.


\bibitem[Fortuin(1980)]{fortuin1980all}
Fortuin, L., 1980. The all‐time requirement of spare parts for service after sales—theoretical analysis and practical results. International Journal of Operations \& Production Management.

\bibitem[Fortuin(1984)]{fortuin1984initial}
Fortuin, L., 1984. Initial supply and re-order level of new service parts. European Journal of Operational Research, 15 (3), 310-319.

\bibitem[Frenk et al.(2019a)]{frenk2019exact}
Frenk, J.B.G., Javadi, S., Pourakbar, M., Sezer, S.O., 2019a. An exact static solution approach for the service parts end-of-life inventory problem. European Journal of Operational Research, 272 (2), 496-504.

\bibitem[Frenk et al.(2019b)]{frenk2019optimal}
Frenk, J.B.G., Javadi, S., Sezer, S.O., 2019b. An optimal stopping approach for the end-of-life inventory problem. Mathematical Methods of Operations Research, 90 (3), 329-363.

\bibitem[Frenk et al.(2019c)]{frenk2019order}
Frenk, J.B.G., Pehlivan, C., Sezer, S.O., 2019c. Order and exit decisions under non-increasing price curves for products with short life cycles. Mathematical Methods of Operations Research, 90 (3), 365-397.

\bibitem[Hur et al.(2018)]{hur2018end}
Hur, M., Keskin, B.B., Schmidt, C.P., 2018. End-of-life inventory control of aircraft spare parts under performance based logistics. International Journal of Production Economics, 204, 186-203.

\bibitem[Jack and Van der Duyn Schouten(2000)]{jack2000optimal}
Jack, N., Van der Duyn Schouten, F., 2000. Optimal repair–replace strategies for a warranted product. International Journal of Production Economics, 67(1), 95-100.

\bibitem[Inderfurth and Kleber(2013)]{inderfurth2013advanced}
Inderfurth, K., Kleber, R., 2013. An advanced heuristic for multiple‐option spare parts procurement after end‐of‐production. Production and Operations Management, 22 (1), 54-70.

\bibitem[Inderfurth and Mukherjee(2008)]{inderfurth2008decision}
Inderfurth, K., Mukherjee, K., 2008. Decision support for spare parts acquisition in post product life cycle. Central European Journal of Operations Research, 16 (1), 17-42.

\bibitem[Jack and Van der Duyn Schouten(2000)]{jack2000optimal}
Jack, N., Van der Duyn Schouten, F., 2000. Optimal repair–replace strategies for a warranted product. International Journal of Production Economics, 67 (1), 95-100.

\bibitem[Kleber et al.(2012)]{kleber2012dynamic}
Kleber, R., Schulz, T., Voigt, G., 2012. Dynamic buy-back for product recovery in end-of-life spare parts procurement. International Journal of Production Research, 50 (6), 1476-1488.

\bibitem[Leifker et al.(2014)]{leifker2014determining}
Leifker, N.W., Jones, P.C., Lowe, T.J., 2014. Determining optimal order amount for end-of-life parts acquisition with possibility of contract extension. The Engineering Economist, 59 (4), 259-281.

\bibitem[Leifker et al.(2012)]{leifker2012continuous}
Leifker, N.W., Jones, P.C., Lowe T.J., 2012. A continuous-time examination of end-of-life parts acquisition with limited customer information. The Engineering Economist, 57 (4), 284-301.

\bibitem[Oh and \"{O}zer(2016)]{oh2016characterizing}
Oh, S., \"{O}zer, \"{O}., 2016. Characterizing the structure of optimal stopping policies. Production and Operations Management, 25 (11), 1820-1838.

\bibitem[Pin\c{c}e et al.(2015)]{pincce2015role}
Pin\c{c}e, \c{C}., Frenk, J.B.G., Dekker, R., 2015. The role of contract expirations in service parts management. Production and Operations Management, 24 (10), 1580-1597.

\bibitem[Pin\c{c}e and Dekker(2011)]{pincce2011inventory}
Pin\c{c}e, \c{C}., Dekker, R., 2011. An inventory model for slow moving items subject to obsolescence. European Journal of Operational Research, 213 (1), 83-95.

\bibitem[Porteus(2002)]{porteus2002foundations}
Porteus, E.L., 2002. Foundations of Stochastic Inventory Theory. Stanford University Press.

\bibitem[Pourakbar et al.(2014)]{pourakbar2014Phaseout}
Pourakbar, M., van der Laan, E., Dekker, R., 2014. End‐of‐life inventory problem with phaseout returns. Production and Operations Management, 23 (9), 1561-1576.

\bibitem[Pourakbar et al.(2012)]{pourakbar2012end}
Pourakbar, M., Frenk, J.B.G., Dekker, R., 2012. End‐of‐life inventory decisions for consumer electronics service parts. Production and Operations Management, 21 (5), 889-906.

\bibitem[Pourakbar and Dekker(2012)]{pourakbar2012customer}
Pourakbar, M., Dekker, R., 2012. Customer differentiated end-of-life inventory problem. European Journal of Operational Research, 222 (1), 44-53.

\bibitem[Shi and Liu(2020)]{shi2020optimal}
Shi, Z., Liu, S., 2020. Optimal inventory control and design refresh selection in managing part obsolescence. European Journal of Operational Research. 287 (1), 133-144. 

\bibitem[Shi(2019)]{shi2019optimal}
Shi, Z., 2019. Optimal remanufacturing and acquisition decisions in warranty service considering part obsolescence. Computers \& Industrial Engineering, 135, 766-779.

\bibitem[Shen and Willems(2014)]{shen2014modeling}
Shen, Y., Willems, S.P., 2014. Modeling sourcing strategies to mitigate part obsolescence. European Journal of Operational Research, 236 (2), 522-533.

\bibitem[Silver et al.(2016)]{silver2016inventory}
Silver, E.A., Pyke, D.F., Thomas, D.J., 2016. Inventory and production management in supply chains. CRC Press.

\bibitem[Teunter and Haneveld(2002)]{teunter2002inventory}
Teunter, R.H., Haneveld, W.K.K., 2002. Inventory control of service parts in the final phase. European Journal of Operational Research, 137 (3), 497-511.

\bibitem[Teunter and Fortuin(1999)]{teunter1999end}
Teunter, R.H., Fortuin, L., 1999. End-of-life service. International Journal of Production Economics, 59(1-3), 487-497.

\bibitem[van der Heijden and Iskandar(2013)]{van2013last}
van der Heijden, M., Iskandar, B.P., 2013. Last time buy decisions for products sold under warranty. European Journal of Operational Research, 224 (2), 302-312.

\bibitem[van Kooten and Tan(2009)]{van2009final}
van Kooten, J.P., Tan, T., 2009. The final order problem for repairable spare parts under condemnation. Journal of the Operational Research Society, 60 (10), 1449-1461.
\end{thebibliography}
\singlespacing

\end{document}